\documentclass[reqno,11pt]{amsart}

\usepackage{amsmath,amsfonts,amssymb,amsthm,epsfig,amscd,}
\usepackage[]{fontenc}
\usepackage[latin1]{inputenc}
\usepackage{enumerate}
\usepackage[cmtip,all]{xy}
\usepackage{tabularx,multirow}
\usepackage{relsize}
\usepackage{layout}
\usepackage{fullpage}
\usepackage{color}
\usepackage{tikz-cd}
 \allowdisplaybreaks 
\numberwithin{equation}{section}
\newtheorem{theorem}{Theorem}[section]

\newtheorem{corollary}[theorem]{Corollary}
\newtheorem{lemma}[theorem]{Lemma}

\newtheorem{claim}[theorem]{Claim}

\theoremstyle{definition}
\newtheorem{definition}[theorem]{Definition}

\newtheorem{remark}[theorem]{Remark}

\DeclareMathOperator{\SP}{SP}

\DeclareMathOperator{\Span}{Span}
\DeclareMathOperator{\Div}{Div}
\DeclareMathOperator{\Supp}{Supp}
\DeclareMathOperator{\covgon}{cov.gon}
\DeclareMathOperator{\conngon}{conn.gon}
\DeclareMathOperator{\gon}{gon}
\DeclareMathOperator{\mult}{mult}
\DeclareMathOperator{\Pic}{Pic}
\DeclareMathOperator{\irr}{irr}

\def\bP{{\mathbb P}}
\def\cE{{\mathcal E}}

\begin{document}

\title[Covering gonality of symmetric products of curves]{Covering gonality of symmetric products of curves and Cayley--Bacharach condition on Grassmannians}

\author{Francesco Bastianelli}
\address{Francesco Bastianelli, Dipartimento di Matematica, Universit\`{a} degli Studi di Bari Aldo Moro, Via Edoardo Orabona 4, 70125 Bari -- Italy}
\email{francesco.bastianelli@uniba.it}

\author{Nicola Picoco}
\address{Nicola Picoco, Dipartimento di Matematica, Universit\`{a} degli Studi di Bari Aldo Moro, Via Edoardo Orabona 4, 70125 Bari -- Italy}
\email{nicola.picoco@uniba.it}

\begin{abstract}
Given an irreducible projective variety $X$, the covering gonality of $X$ is the least gonality of an irreducible curve $E\subset X$ passing through a general point of $X$.
In this paper we study the covering gonality of the $k$-fold symmetric product $C^{(k)}$ of a smooth complex projective curve $C$ of genus $g\geq k+1$.
It follows from a previous work of the first author that the covering gonality of the second symmetric product of $C$ equals the gonality of $C$.
Using a similar approach, we prove the same for the $3$-fold and the $4$-fold symmetric product of $C$.

A crucial point in the proof is the study of the Cayley--Bacharach condition on Grassmannians.
In particular, we describe the geometry of linear subspaces of $\mathbb{P}^n$ satisfying this condition and we prove a result bounding the dimension of their linear span.
\end{abstract}

\thanks{This work was partially supported by INdAM (GNSAGA)} 
%\keywords{Cayley--Bacharach property; symmetric product of curves; covering gonality; measures of irrationality; Grassmannians}
%\subjclass[2020]{Primary 14E99 14H10 Secondary 14E08 14H51 14N05}

\maketitle

\section{Introduction}\label{section:intro}

In recent years there has been a considerable amount of interest and work concerning covering gonality and, more generally, measures of irrationality of projective varieties (see e.g. \cite{BDELU, BCFS, CS, CMNP, LM, Mar, V}).
Along these lines, we are interested in determining the covering gonality of symmetric products of curves.
To this aim, we also come across the study of the Cayley--Bacharach condition on Grassmannians, a topic of independent interest, and we describe the geometry of linear subspaces of $\mathbb{P}^n$ satisfying this condition.

\smallskip
Let $X$ be an irreducible complex projective variety.
We recall that the \emph{gonality} of an irreducible complex projective curve $E$, denoted by $\gon(E)$, is the least degree of a non-constant morphism $\widetilde{E}\longrightarrow \mathbb{P}^1$, where $\widetilde{E}$ is the normalization of $E$.
The \emph{covering gonality} of $X$ is the birational invariant defined as
\begin{displaymath}
\covgon(X):=\min\left\{d\in \mathbb{N}\left|
\begin{array}{l}
\text{Given a general point }x\in X,\,\exists\text{ an irreducible}\\ \text{curve } E\subseteq X  \text{ such that }x\in E \text{ and }\gon(E)=d
\end{array}\right.\right\}.
\end{displaymath}
In particular, if $X$ is a curve, then $\covgon(X)$ equals $\gon(X)$.
Furthermore, it is worth noticing that $\covgon(X)=1$ if and only if $X$ is uniruled.
Thus the covering gonality of $X$ can be viewed as a measure of the failure of $X$ to be uniruled.
We are aimed at computing this invariant when $X$ is the $k$-fold symmetric product of a smooth curve.
 
\smallskip
Let $C$ be a smooth complex projective curve of genus $g$ and let $k$ be a positive integer.
The $k$-fold symmetric product of $C$ is the smooth projective variety $C^{(k)}$ parameterizing unordered $k$-tuples $p_1+\dots+p_k$ of points of $C$.
This is an important variety reflecting the geometry of $C$, which is involved in many fundamental results on algebraic curves, as e.g. in Brill--Noether theory.

We note that $C^{(k)}$ is covered by copies of $C$ of the form $C_P:=\left\{\left.P+q\in C^{(k)}\right|q\in C\right\}$, where $P=p_1+\dots+p_{k-1}\in C^{(k-1)}$ is a fixed point. 
Hence we deduce the obvious bound 
\begin{equation}
\label{eq:bound intro}
\covgon\big(C^{(k)}\big)\leq\gon(C).
\end{equation}
Therefore the main issue is bounding the covering gonality from below. 
In \cite{B1}, the first author proved that the covering gonality of $C^{(2)}$ equals the gonality of $C$, i.e. \eqref{eq:bound intro} is actually an equality, provided that $g\geq 3$.

In this paper, we prove the same for the $3$-fold and the $4$-fold symmetric product of a curve.
\begin{theorem}\label{thm:covgon}
Let $k\in\{3,4\}$ and let $C$ be a smooth complex projective curve of genus $g\geq k+1$.
Then the covering gonality of the $k$-fold symmetric product of $C$ is
\begin{equation}\label{eq:covgon}
\covgon\big(C^{(k)}\big)=\gon(C),
\end{equation}
provided that $\big(k,g,\gon(C)\big)\not\in\left\{ (3,4,3),(4,5,4)\right\}$.
%$\big(k,g,\gon(C)\big)\neq (k,k+1,k)$, 

\end{theorem}

%\begin{theorem}\label{thm:covgon}
%Let $C$ be a smooth complex projective curve of genus $g$. 
%\begin{itemize}
%  \item If $g\geq 4$ and $(g,\gon(C))\neq (4,3)$, then $$\covgon\big(C^{(3)}\big)=\gon(C).$$
%  \item If $g\geq 5$ and $(g,\gon(C))\neq (5,4)$, then $$\covgon\big(C^{(4)}\big)=\gon(C).$$ 
%\end{itemize}
%\end{theorem}

Furthermore, a problem which naturally arises from \cite[Theorem 1.6]{B1} and Theorem \ref{thm:covgon} is characterizing families of irreducible curves on $C^{(k)}$ which compute its covering gonality.
This is in fact the purpose of \cite{BP2}, where we use similar techniques, and we prove that if $2\leq k\leq 4$ and $C$ is a sufficiently general curve of genus $g\geq k+4$, then the curves $C_P\subset C^{(k)}$ defined above are the only irreducible curves covering $C^{(k)}$ and having the same gonality as $C$.

Concerning the exceptional cases $(k,g,\gon(C))\in\{(3,4,3),(4,5,4)\}$ in the statement of Theorem \ref{thm:covgon}, we can not decide whether the covering gonality of the symmetric product equals $\gon(C)$ or $\gon(C)-1$. 
Moreover, we discuss the cases of low genus in Remark \ref{rem:low genus}. 

In the light of \cite[Theorem 1.6]{B1} and Theorem \ref{thm:covgon}, which describe $\covgon(C^{(k)})$ when $g\geq k+1$ and $k=2,3,4$, it seems plausible that the equality $\covgon\big(C^{(k)}\big)=\gon(C)$ may hold for any $2\leq k\leq g-1$.
On the other hand, we note that if $k\geq g+1$, the $k$-fold symmetric product of $C$ is covered by rational curves, so that $\covgon(C^{(k)})=1$ (cf. Remark \ref{rem:low genus}).
Finally, the remaining case $k=g$ is quite intriguing, since $C^{(g)}$ is birational to the Jacobian variety of $C$.
In this setting, inequality \eqref{eq:bound intro} fails to be an equality for some special curves, but the behavior of $\covgon(C^{(g)})$ for low genus suggests that equality \eqref{eq:covgon} might hold when $C$ has very general moduli (see Remark \ref{rem:low genus}).

\smallskip
In order to prove Theorem \ref{thm:covgon}, we follow the same argument as \cite{B1}.
In particular, let $\phi\colon C\longrightarrow \mathbb{P}^{g-1}$ be the canonical map and let $\gamma\colon C^{(k)}\dashrightarrow \mathbb{G}(k-1,g-1)$ be the \emph{Gauss map} of $C^{(k)}$, which sends a general point $P=p_1+\dots+p_{k}\in C^{(k)}$ to the point of the Grassmannian parameterizing the $(k-1)$-plane $\Span\left(\phi(p_1),\dots,\phi(p_k)\right)\subset \mathbb{P}^{g-1}$.
Moreover, let $E\subset C^{(k)}$ be a $d$-gonal curve passing through a general point of $C^{(k)}$, and let $P_1,\dots,P_d\in E$ be the points corresponding to a general fiber of the $d$-gonal map $\widetilde{E}\longrightarrow \mathbb{P}^1$.
According to the framework of \cite[Section 4]{B1}, the points $\gamma\left(P_1\right),\ldots, \gamma\left(P_d\right)\in \mathbb{G}(k-1,g-1)$ satisfy a condition of Cayley--Bacharach type, that we describe below (see also Theorem \ref{thm:covfamCk}).
Then a crucial point in the proof is showing that the corresponding $(k-1)$-planes in $\mathbb{P}^{g-1}$ span a linear space of sufficiently small dimension.
In particular, this governs the dimension of the linear series on $C$ defined by the points of $\phi(C)$ supporting $P_1,\dots,P_d\in C^{(k)}$.
So, we need a large amount of work based on standard results in Brill--Noether theory, in order to discuss existence of linear series on $C$, and we eventually conclude that $d\geq \gon(C)$. 

Unfortunately, the combinatorics of this approach does not work for $k\geq 5$.

\medskip
Turning to the Cayley--Bacharach condition, we consider a finite set of points $\Gamma\subset \mathbb{P}^n$ and a positive integer $r$.
We recall that $\Gamma$ satisfies the Cayley--Bacharach condition with respect to hypersurfaces of degree $r$ if any hypersurface passing through all but one point of $\Gamma$, passes through the last point too.
This is a very classical property, whose history goes back even to ancient geometry (see \cite{EGH} for a detailed treatise on this topic).
Besides, the Cayley--Bacharach condition has also been studied in recent years, and it has been applied to several issues, concerning e.g. Brill--Noether theory of space curves and measures of irrationality of algebraic varieties (cf. \cite{LP, BCD, LU, Pic, SU}).

More generally, we consider the Grassmannian $\mathbb{G}=\mathbb{G}(k-1,n)$ parameterizing $(k-1)$-dimensional linear subspaces of $\mathbb{P}^n$, and we say that a finite set of points $\Gamma\subset \mathbb{G}$ \emph{satisfies the Cayley--Bacharach condition with respect to} $\left|\mathcal{O}_{\mathbb{G}}(r)\right|$ if any effective divisor of $\left|\mathcal{O}_{\mathbb{G}}(r)\right|$ passing through all but one point of $\Gamma$, passes through the last point too.
Of course, the case $k=1$ recovers the former definition for points in $\mathbb{P}^n$.

We point out further that when $r=1$, if $\Gamma=\left\{[\Lambda_1],\dots,[\Lambda_d]\right\}\subset \mathbb{G}$ satisfies the Cayley--Bacharach condition with respect to $\left|\mathcal{O}_{\mathbb{G}}(1)\right|$, then any $(n-k)$-plane $L\subset \mathbb{P}^n$ intersecting all but one $(k-1)$-plane parameterized by $\Gamma$, must intersect the last one too (cf. Remark \ref{rem:SP}).
In this case we say that the $(k-1)$-planes $\Lambda_1,\dots,\Lambda_d\subset \mathbb{P}^n$ \emph{are in special position} (\emph{with respect to $(n-k)$-planes}); see Definition \ref{def:SP}.
It is worth noting that the study of the geometry of lines in special position was crucial in \cite{B1} and \cite{GK}, in order to describe measures of irrationality of $C^{(2)}$ and of the Fano surface of cubic threefolds, respectively. 

In this paper, we improve the results in \cite[Section 2]{B1} on the linear span of $(k-1)$-planes in special position.
In particular, we prove the following (cf. Theorem \ref{thm:dimSP}).

\begin{theorem}\label{thm:dimSP intro}
Let $\Gamma=\left\{\Lambda_1,\dots,\Lambda_d\right\}\subset \mathbb{P}^{n}$ be a set of distinct $(k-1)$-dimensional linear subspaces in special position with respect to $(n-k)$-planes.
Assume further that $\Gamma$ does not admit a partition in subsets, whose elements are still in special position. 
Then
$$
\dim\Span\left(\Lambda_1,\dots,\Lambda_d\right)\leq d+k-3.
$$
\end{theorem}
In particular, the theorem highlights that being in special position imposes strong restrictions to the geometry of the $k$-planes $\Lambda_1,\dots,\Lambda_d$, since $d$ linear spaces of dimension $k$ in \emph{general} position span a linear space of dimension $d(k+1)-1$.
We refer to Theorem \ref{thm:dimSP} and Corollary \ref{cor:dimSP} for more general statements, where the $(k-1)$-planes are allowed to coincide and $\Gamma$ may admit partitions as above.

Apart from the dimension of the span of linear spaces in special position, there are other interesting issues concerning the Cayley--Bacharach property on Grassmannians and its applications; we briefly mention them in \S \ref{sub:remarks}. 

The proof of Theorem \ref{thm:dimSP intro} is rather long and relies on various simple properties of linear spaces in special position. 
The idea is to consider the images of the planes $\Lambda_1,\dots,\Lambda_{d-1}$ under the projection $\pi_p\colon \mathbb{P}^n\dashrightarrow \mathbb{P}^{n-1}$ from a general point $p\in \Lambda_d$.
If $\Gamma'=\left\{\pi_p(\Lambda_1),\dots,\pi_p(\Lambda_{d-1})\right\}\subset \mathbb{P}^{n-1}$ does not admit a partition in subsets of $(k-1)$-planes in special position, the assertion follows easily by induction. 
It takes a great deal of work to prove the assertion when $\Gamma'$ admits such a partition.

\smallskip
The rest of the paper consists of two parts.
In Section \ref{section:sp}, we discuss the Cayley--Bacharach condition on Grassmannians and we prove our results on linear subspaces of $\mathbb{P}^n$ in special positions.

In Section \ref{section:covgonCk}, we are instead concerned with the covering gonality of $C^{(k)}$. 
Initially, we recall preliminary notions on covering gonality and we recollect the framework of \cite{B1}.
Then we prove Theorem \ref{thm:covgon}.

\subsection*{Notation}

We work throughout over the field $\mathbb{C}$ of complex numbers.
By \emph{variety} we mean a complete reduced algebraic variety $X$, and by \emph{curve} we mean a variety of dimension 1.
When we speak of a \emph{smooth} curve, we always implicitly assume it to be irreducible.
We say that a property holds for a \emph{general} (resp. \emph{very general}) point ${x\in X}$ if it holds on a Zariski open nonempty subset of $X$ (resp. on the complement of the countable union of proper subvarieties of $X$).

%%%%%%%%%%%%%%%%%%%%%%%%%%%%%%%%%%%%%%%
%%%%%%%%%%%%%%%%%%%%%%%%%%%%%%%%%%%%%%%
%%%%%%%%%%%%%%%%%%%%%%%%%%%%%%%%%%%%%%%

\medskip
\section{Cayley--Bacharach condition on Grassmannians}\label{section:sp} 

In this section, we are interested in the Cayley--Bacharach condition on Grassmannians.
In particular, we discuss linear subspaces of $\mathbb{P}^n$ in special position in \S \ref{sub:sp}, and we prove our results on the dimension of their linear span in \S \ref{sub:proof dimSP}.
Finally, we present in \S \ref{sub:remarks} further remarks and possible developments about the Cayley--Bacharach condition on Grassmannians.

\smallskip
\subsection{Linear subspaces of $\mathbb{P}^n$ in special position}\label{sub:sp} 
Let us fix two integers $1\leq k\leq n$. 
In this section, we are concerned with collections of $(k-1)$-dimensional linear spaces in $\mathbb{P}^n$ satisfying a property of Cayley-Bacharach type, defined as follows (cf. \cite[Definition 3.2]{B1}).
\begin{definition}\label{def:SP}
Let $\Lambda_1,\dots,\Lambda_d\subset \mathbb{P}^{n}$ be linear subspaces of dimension $k-1$.
We say that $\Lambda_1,\dots,\Lambda_d$ are \emph{in special position with respect to} $(n-k)$\emph{-planes}, or just $\SP(n-k)$, if for any $j=1,\dots,d$ and for any $(n-k)$-plane $L\subset \mathbb{P}^n$ intersecting $\Lambda_1,\dots,\widehat{\Lambda_j},\dots,\Lambda_d$, we have that $L$ meets $\Lambda_j$ too.
\end{definition}

Along the same lines as \cite[Section 3]{B1}, we are aimed at bounding the dimension of the linear span of collections of $(k-1)$-planes satisfying property $\SP(n-k)$.

\begin{remark}\label{rem:SP}
The connection between special position and the Cayley-Bacharach property may be expressed as follows.
We recall that, given a complete linear series $\mathcal{D}$ on a variety $X$, we say that $d$ points $P_1,\dots,P_d\in X$ satisfy the \emph{Cayley-Bacharach condition with respect to} $\mathcal{D}$ if for any $j=1,\dots,d$ and for any effective divisor $D\subset \mathcal{D}$ passing through $P_1,\dots,\widehat{P_j},\dots,P_d$, we have $P_j\in D$ too.

When $X=\mathbb{G}(k-1,n)$ is the Grassmannian of $(k-1)$-planes in $\mathbb{P}^n$, for any $(n-k)$-dimensional linear subspace $L\subset \mathbb{P}^n$, the Schubert cycle $\sigma_1(L):=\left\{[\Lambda]\in X|\Lambda\cap L\neq \emptyset\right\}$ is an effective divisor of $\left|\mathcal{O}_X(1)\right|$.
Thus if $[\Lambda_1],\dots,[\Lambda_d]\in X$ is a collection of points satisfying the Cayley-Bacharach condition with respect to $\left|\mathcal{O}_X(1)\right|$, then the corresponding $(k-1)$-planes $\Lambda_1,\dots,\Lambda_d\subset \mathbb{P}^{n}$ are $\SP(n-k)$.

Furthermore, if the collection $[\Lambda_1],\dots,[\Lambda_d]\in X$ satisfies the Cayley-Bacharach condition with respect to $\left|\mathcal{O}_X(r)\right|$ for some $r\geq 1$, the same holds with respect to $\left|\mathcal{O}_X(s)\right|$ for any $1\leq s\leq r$ and, in particular, the $(k-1)$-planes $\Lambda_1,\dots,\Lambda_d\subset \mathbb{P}^{n}$ are $\SP(n-k)$.
\end{remark}

\begin{remark}\label{rem:unionSP}
We note that if the $(k-1)$-planes $\Lambda_1,\dots,\Lambda_d$ are $\SP(n-k)$, then $d\geq 2$.
Moreover, the linear spaces $\Lambda_1,\dots,\Lambda_d$ are not necessarily distinct.
In particular, two $(k-1)$-planes $\Lambda_1,\Lambda_2$ are $\SP(n-k)$ if and only if they coincide.

Furthermore, if $\Lambda_1,\dots,\Lambda_d$ and $\Lambda'_{1},\dots,\Lambda'_e$ are two collections of $(k-1)$-planes that satisfy $\SP(n-k)$, then $\Lambda_1,\dots,\Lambda_d,\Lambda'_{1},\dots,\Lambda'_e$ are $\SP(n-k)$.
\end{remark}

We are now interested in sequences of $(k-1)$-planes satisfying property $\SP(n-k)$, that can not be partitioned in subsequences of $(k-1)$-planes that satisfy $\SP(n-k)$.
\begin{definition}\label{def:partition}
Let $\Lambda_1,\dots,\Lambda_d\subset \mathbb{P}^{n}$ be linear subspaces of dimension $k-1$ in special position with respect to $(n-k)$-planes.
We say that the sequence $\Lambda_1,\dots,\Lambda_d$ is \emph{decomposable} if there exists a partition of the set of indexes $\{1,\dots,d\}$, where each part $\{i_1,\dots,i_t\}$ is such that the corresponding $(k-1)$-planes $\Lambda_{i_1},\dots,\Lambda_{i_t}$ are $\SP(n-k)$.
Otherwise, we say that the sequence is \emph{indecomposable}.
\end{definition}

Then we can state the main result of this section, which bounds the dimension of the linear span of an indecomposable sequence of $(k-1)$-planes satisfying property $\SP(n-k)$.

\begin{theorem}\label{thm:dimSP}
Let $\Lambda_1,\dots,\Lambda_d\subset \mathbb{P}^{n}$ be an indecomposable sequence of linear subspaces of dimension $k-1$ in special position with respect to $(n-k)$-planes.
Then
$$
\dim\Span\left(\Lambda_1,\dots,\Lambda_d\right)\leq d+k-3.
$$
\end{theorem}

Before proving the theorem, we deduce the following corollary, which refines the bound given in \cite[Theorem 3.3]{B1} on the dimension of the linear span of a sequence of $(k-1)$-planes satisfying $\SP(n-k)$.
Furthermore, this result shall play a crucial role in computing the covering gonality of symmetric products of curves.

\begin{corollary}\label{cor:dimSP}
Let $\Lambda_1,\dots,\Lambda_d\subset \mathbb{P}^{n}$ be linear subspaces of dimension $k-1$ in special position with respect to $(n-k)$-planes.
Consider a partition of the set of indexes $\{1,\dots,d\}$ such that each part $\{i_1,\dots,i_t\}$ corresponds to an indecomposable sequence $\Lambda_{i_1},\dots,\Lambda_{i_t}$ of $(k-1)$-planes that satisfy $\SP(n-k)$.
Denoting by $m$ the number of parts of the partition, we have that
$$
\dim\Span\left(\Lambda_1,\dots,\Lambda_d\right)\leq d+k-3+(m-1)(k-2).
$$
\begin{proof}
For any $i=1,\dots m$, let $d_i$ be the number of integers contained in the $i$-th part, so that $d_1+\dots+d_m=d$.
Since any part indexes an indecomposable sequence of $d_i$ linear spaces satisfying $\SP(n-k)$, Theorem \ref{thm:dimSP} ensures that the span of those linear spaces has dimension at most $d_i+k-3$.
Thus 
$$\dim\Span\left(\Lambda_{1},\dots,\Lambda_{d}\right)  \leq \left(d_1+k-3\right) + \sum_{i=2}^m \left(d_i+k-2\right)= d+k-3+(m-1)(k-2),$$
as claimed.
\end{proof}
\end{corollary}

\begin{remark}
A partition as in Corollary \ref{cor:dimSP} always exists. In particular, if the sequence is indecomposable, we have the trivial partition $\{1,\dots,d\}$ and we recover Theorem \ref{thm:dimSP}.
Furthermore, the partition is not necessarily unique, and the best bound is given by partitions having the least number of parts.
\end{remark}

\begin{remark}
If $k\geq 2$ and $\Lambda_1,\dots,\Lambda_d\subset \mathbb{P}^{n}$ are $(k-1)$-planes that satisfy $\SP(n-k)$, then \cite[Theorem 3.3]{B1} asserts that $\dim \Span(\Lambda_1,\dots,\Lambda_d)\leq \lfloor \frac{kd}{2}\rfloor-1$.
Thus Theorem \ref{thm:dimSP} gives a strong improvement of this bound when the sequence $\Lambda_1,\dots,\Lambda_d$ is indecomposable and $k\geq 3$.

It is also easy to see that Corollary \ref{cor:dimSP} improves \cite[Theorem 3.3]{B1}, apart from the case $k=2$ and for few exceptional configurations of the $(k-1)$-planes $\Lambda_i$, where the two bounds coincide. 
Moreover, both Theorem \ref{thm:dimSP} and Corollary \ref{cor:dimSP} turn out to be sharp (see e.g. \cite[Examples 3.5 and 3.6]{B1}).
\end{remark}

\smallskip
\subsection{Proof of Theorem \ref{thm:dimSP}}\label{sub:proof dimSP}

The proof of Theorem \ref{thm:dimSP} is quite long and relies on various simple properties of linear spaces in special position. 
The idea is to consider the images of the planes $\Lambda_1,\dots,\Lambda_{d-1}$ under the projection $\pi_p\colon \mathbb{P}^n\dashrightarrow \mathbb{P}^{n-1}$ from a general point $p\in \Lambda_d$.
If $\pi_p(\Lambda_1),\dots,\pi_p(\Lambda_{d-1})\subset \mathbb{P}^{n-1}$ is an indecomposable sequence of $(k-1)$-planes satisfying $\SP(n-k)$, the assertion follows easily by induction. 
On the other hand, we need a big amount of work to prove the assertion when the sequence of linear spaces $\pi_p(\Lambda_i)$ is decomposable.

So, in order to prove Theorem \ref{thm:dimSP}, we firstly present various preliminary lemmas, which express properties of linear spaces in special position.

\begin{lemma}\label{lem:B1}
Let $\Lambda_1,\dots,\Lambda_d\subset \mathbb{P}^{n}$ be $(k-1)$-planes that satisfy $\SP(n-k)$.
If $\Lambda_1,\dots,\Lambda_{d-1}$ are contained in a linear space $S$, then $\Lambda_d\subset S$ too.
\begin{proof}
See \cite[Lemma 3.4]{B1}.
\end{proof}
\end{lemma}

\begin{lemma}\label{lem:BVT}
Let $\Lambda_1,\dots,\Lambda_d\subset\mathbb{P}^n$ be $(k-1)$-planes that satisfy $\SP(n-k)$. 
Let $1 \leq r \leq d-2$ be an integer and let $(p_{r+1},\dots,p_{d-1})\in \Lambda_{r+1}\times\dots\times\Lambda_{d-1}$ be a $(d-1-r)$-tuple such that 
$$\Span(p_{r+1},\dots,p_{d-1})\cap \Lambda_d = \emptyset.$$
Then $\Lambda_d\subset\Span(\Lambda_1,\dots,\Lambda_r ,p_{r+1},\dots,p_{d-1})$. 

\begin{proof} Let $S:=\Span(\Lambda_1,\dots,\Lambda_r ,p_{r+1},\dots,p_{d-1})$ and let $s:=\dim S$. 
If $S$ is the whole projective space $\mathbb{P}^n$ there is nothing to prove, hence we assume hereafter $s < n$. 
Moreover, $s \geq k - 1$ and hence $0 \leq s - k + 1 \leq n - k$.

Let $P:=\Span(p_{r+1},\dots,p_{d-1})\subset S$. 
By assumption, $P$ does not meet $\Lambda_d$, therefore $\dim P\leq n-k$. 
Furthermore, $\dim P \leq s -k$. 
Indeed, if the dimension of $P$ were greater than $s - k$, we would have $P \cap \Lambda_j\neq \emptyset$ for any $j=1,\dots,r$. 
Therefore $P$ would meet each $(k - 1)$-plane $\Lambda_j$, with $j=1,\dots,d-1$. 
Because of condition $\SP(n-k)$, any $(n-k)$-plane $L \subset\mathbb{P}^n$ containing $P$ should intersect also $\Lambda_d$. 
Then we would have $P\cap \Lambda_d\neq\emptyset$, a contradiction.

Now, let $T \subset S$ be a $(s -k + 1)$-plane containing $P$. 
Then $T$ meets each of the $(k-1)$-planes $\Lambda_1,\dots,\Lambda_{d-1}$. 
By condition $\SP(n-k)$, any $(n-k)$-plane $L$ containing $T$ must intersect the remaining plane $\Lambda_d$. 
Therefore $\Lambda_d\cap T\neq \emptyset$ for any $(s -k +1)$-plane $T$ as above. 
Finally, as $\Lambda_d \cap P =\emptyset$ by assumption, we deduce that $\Lambda_d\subset S$. 
\end{proof}
\end{lemma}

\begin{lemma}\label{lem:intersezioneNonSP}
Let $\Lambda_1,\dots,\Lambda_d\subset\mathbb{P}^n$ be $(k-1)$-planes that satisfy $\SP(n-k)$. 
Let $2\leq r \leq d-1$ be an integer such that $\Lambda_{r+1},\dots,\Lambda_d$ are not $\SP(n-k)$. 
Then
$$\dim\big(\Span(\Lambda_1,\dots,\Lambda_r)\cap\Span(\Lambda_{r+1},\dots,\Lambda_d)\big)\geq k-1.$$

\begin{proof} 
By Lemma \ref{lem:B1}, the $(k-1)$-plane $\Lambda_d$ is contained in $\Span(\Lambda_1,\dots,\Lambda_{d-1})$. Hence the assertion holds for $r=d-1$.
 
So we assume $2\leq r \leq d-2$.
Since $\Lambda_{r+1},\dots,\Lambda_d$ are not $\SP(n-k)$, there exists a $(n-k)$-plane $L$ intersecting all the $(k-1)$-planes $\Lambda_{r+1},\dots,\Lambda_d$ except one, say $\Lambda_d$. 
For any $j=r+1,\dots,d-1$, we fix a point $p_j\in L\cap \Lambda_j$.
Let us define $P:=\Span(p_{r+1},\dots,p_{d-1})\subset L$ and $S:=\Span(\Lambda_{1},\dots,\Lambda_r,p_{r+1},\dots,p_{d-1})$.
Hence $P\cap \Lambda_d=\emptyset$, and Lemma \ref{lem:BVT} ensures that $\Lambda_d\subset S$. 

Setting $\alpha:=\dim P$ and $\beta:=\dim\Span(\Lambda_{1},\dots,\Lambda_r)$, we have that $\dim\Span(P,\Lambda_d)=\alpha+k$ and $\dim S\leq\alpha+\beta+1$.
Since $\Span(\Lambda_{1},\dots,\Lambda_r)$ and $\Span(P,\Lambda_d)$ are contained in $S$, we obtain
$$\dim \big(\Span(\Lambda_{1},\dots,\Lambda_r)\cap \Span(P,\Lambda_d)\big)\geq \beta+(\alpha+k)-(\alpha+\beta+1)=k-1.$$
Finally, as $\Span(P,\Lambda_d)\subset \Span(\Lambda_{r+1},\dots,\Lambda_d)$, the assertion follows. 
\end{proof}
\end{lemma}

Given a linear subspace $P\subset \mathbb{P}^n$ of dimension $\alpha:=\dim P\geq 0$, we denote by $\pi_P\colon\mathbb{P}^n \dashrightarrow \mathbb{P}^{n-\alpha-1}$ the projection from $P$, which is defined on $\mathbb{P}^n\smallsetminus P$.
Hence, if $\Lambda\subset \mathbb{P}^n$ is a linear subspace with $\dim(P\cap \Lambda)=\delta$, then $\dim \pi_P(\Lambda)=\dim \Lambda-\delta-1$, whereas for any linear subspace $R\subset \mathbb{P}^{n-\alpha-1}$, the Zariski closure of $\pi_P^{-1}(R)\subset \mathbb{P}^n$ is a linear subspace of dimension $\dim R+\alpha+1$ which contains $P$. 

\begin{lemma}\label{lem:sottosopraSP}
Let $\Lambda_1,\dots,\Lambda_d\subset\mathbb{P}^n$ be $(k-1)$-planes and let $P\subset \mathbb{P}^n$ be a linear subspace of dimension $\alpha\geq 0$ such that $P\cap \Span \left(\Lambda_1,\dots,\Lambda_d\right)=\emptyset$.
Using the notation above, consider the $(k-1)$-planes $\Gamma_i:=\pi_P(\Lambda_i)\subset \mathbb{P}^{n-\alpha-1}$, with $i=1,\dots,d$. 
Then $\Lambda_1,\dots,\Lambda_d\subset\mathbb{P}^n$ are $\SP(n-k)$ if and only if $\Gamma_1,\dots,\Gamma_d$ are $\SP(n-\alpha-1-k)$.

\begin{proof}
Suppose that $\Lambda_1,\dots,\Lambda_d\subset\mathbb{P}^n$ are $\SP(n-k)$ and let $R\subset  \mathbb{P}^{n-\alpha-1}$ be a $(n-\alpha-1-k)$-plane intersecting $\Gamma_1,\dots,\widehat{\Gamma_j},\ldots,\Gamma_{d}$, with $j\in\left\{1,\dots,d\right\}$.
Then the linear subspace $L:=\overline{\pi_P^{-1}(R)}\subset \mathbb{P}^n$ is a $(n-k)$-plane, which intersects $\Lambda_1,\dots,\widehat{\Lambda_j},\ldots,\Lambda_d$. 
Hence $L$ must meet $\Lambda_j$ too, so that $R=\pi_P(L)$ intersects $\Gamma_j=\pi_P(\Lambda_j)$. 
In conclusion, $\Gamma_1,\dots,\Gamma_d$ are $\SP(n-\alpha-1-k)$. 

Vice versa, suppose that $\Gamma_1,\dots,\Gamma_d$ are $\SP(n-\alpha-1-k)$ and let $L\subset \mathbb{P}^n$ be a $(n-k)$-plane intersecting $\Lambda_1,\dots,\widehat{\Lambda_j},\ldots,\Lambda_d$, with $j\in\left\{1,\dots,d\right\}$. 
For any $i=1,\dots,\widehat{j},\ldots,d$, we fix a point $p_i\in\Lambda_i\cap L$ and we define $N:=\Span(p_1,\dots,\widehat{p_j},\ldots,p_d)\subset L$.
Since $N\cap P=\emptyset$, the linear space $\pi_P(N)\subset \mathbb{P}^{n-\alpha-1}$ has dimension $\dim N\leq \dim L=n-k$, and it meets $\Gamma_1,\dots,\widehat{\Gamma_j},\ldots,\Gamma_{d}$.

We claim that $\pi_P(N)$ must intersect $\Gamma_j$ too.
Indeed, if $\dim \pi_P(N)\leq n-\alpha-1-k$, condition $\SP(n-\alpha-1-k)$ yields that any $(n-\alpha-1-k)$-plane containing $\pi_P(N)$ must meet $\Gamma_j$, so that $\pi_P(N)\cap \Gamma_j\neq \emptyset$.
If instead $\dim \pi_P(N)\geq n-\alpha-k$, then $\pi_P(N)$ meets $\Gamma_j\subset \mathbb{P}^{n-\alpha-1}$ as $\dim\Gamma_j=k-1$.

It follows that $\widetilde{N}:=\overline{\pi_P^{-1}\left(\pi_P(N)\right)}$ intersects $\Lambda_j$.
Denoting $S:= \Span \left(\Lambda_1,\dots,\Lambda_d\right)$, we have $P\cap S=\emptyset$ by assumption, so the restriction $\pi_{P}|_S\colon S\longrightarrow \mathbb{P}^{n-\alpha-1}$ maps isomorphically to its image. 
Thus $\widetilde{N}\cap S=N$, as $\pi_P(\widetilde{N}\cap S)=\pi_P(N)$.
Since $\widetilde{N}\cap \Lambda_j\neq \emptyset$ and $\Lambda_j\subset S$, we deduce that $N\subset L$ intersects $\Lambda_j$, and we conclude that $\Lambda_1,\dots,\Lambda_d\subset\mathbb{P}^n$ are $\SP(n-k)$.
\end{proof}
\end{lemma}

\begin{lemma}\label{lem:sottoSP}
Let $\Lambda_1,\dots,\Lambda_d\subset\mathbb{P}^n$ be $(k-1)$-planes that satisfy $\SP(n-k)$ and let $P\subset \mathbb{P}^n$ be a linear subspace of dimension $\alpha\geq 0$. 
Let $2\leq r\leq d-1$ be an integer such that $P\cap \Lambda_i=\emptyset$ for any $i=1,\dots,r$, and $P\cap \Lambda_i\neq \emptyset$ for any $i=r+1,\dots,d$.
Using the notation above, consider the $(k-1)$-planes $\Gamma_i:=\pi_P(\Lambda_i)\subset \mathbb{P}^{n-\alpha-1}$, with $i=1,\dots,r$. 
Then $\Gamma_1,\dots,\Gamma_r$ are $\SP(n-\alpha-1-k)$.

\begin{proof}
By arguing as in the previous proof, we consider a $(n-\alpha-1-k)$-plane $R\subset  \mathbb{P}^{n-\alpha-1}$ intersecting $\Gamma_1,\dots,\widehat{\Gamma_j},\ldots,\Gamma_{r}$, with $j\in\left\{1,\dots,d\right\}$.
Then the linear subspace $L:=\overline{\pi_P^{-1}(R)}\subset \mathbb{P}^n$ is a $(n-k)$-plane, which intersects $\Lambda_1,\dots,\widehat{\Lambda_j},\ldots,\Lambda_r$ and meets $\Lambda_{r+1},\dots,\Lambda_d$ at $P$. 
Hence $L$ must intersect $\Lambda_j$ too, since $\Lambda_1,\dots,\Lambda_d$ are $\SP(n-k)$. Therefore $R=\pi_P(L)$ intersects $\Gamma_j=\pi_P(\Lambda_j)$ and we conclude that $\Gamma_1,\dots,\Gamma_r$ are $\SP(n-\alpha-1-k)$. 
\end{proof}
\end{lemma}

\begin{lemma}\label{lem:sopraSP}
Let $\Lambda_1,\dots,\Lambda_d\subset\mathbb{P}^n$ be $(k-1)$-planes that satisfy $\SP(n-k)$.
Let $2\leq r\leq \min\{d-2,n-k+1\}$ be an integer such that if $p\in \Lambda_{d}$ is a general point, then $p\not\in \Lambda_i$ for $i=1,\dots,r$ and the $(k-1)$-planes $\Gamma_1,\dots,\Gamma_{r}\subset \mathbb{P}^n$ are $\SP(n-k-1)$, where $\pi_p\colon \mathbb{P}^n\dashrightarrow \mathbb{P}^{n-1}$ is the projection from $p$ and $\Gamma_i:=\pi_p(\Lambda_i)$ for $i=1,\dots,r$.
Then one of the following holds:
\begin{itemize}
\item[(A)] $\Lambda_{d}\subset \Span(\Lambda_1,\dots,\Lambda_{r})$;
\item[(B)] $\Lambda_1,\dots,\Lambda_{r}$ are $\SP(n-k)$.
\end{itemize}

\begin{proof}
Let $S:=\Span(\Lambda_1,\dots,\Lambda_{r})$.
If condition (A) does not hold, then $S\cap \Lambda_{d}$ is a (possibly empty) proper subset of $\Lambda_{d}$. 
Suppose that $L$ is a $(n-k)$-plane intersecting $\Lambda_1,\dots,\widehat{\Lambda_j},\ldots,\Lambda_{r}$. 
For any $i=1,\dots,\widehat{j},\dots,r$, let us fix a point $p_i\in L\cap\Lambda_i$ and set $N:=\Span\left(p_1,\dots,\widehat{p_j},\dots,p_r\right)\subset L\cap S$.

Notice that a general point $p\in\Lambda_{d}$ is such that $p\not\in S$, and consider $\pi_p(N)\subset \mathbb{P}^{n-1}$. 
Then $\dim \pi_p(N)=\dim N\leq r-2 \leq n-k-1$. 
By definition of $N$, the subspace $\pi_p(N)$ intersects $\Gamma_1,\dots,\widehat{\Gamma_j},\ldots,\Gamma_{r}$ and by condition $\SP(n-k-1)$, $\pi_p(N)$ must intersect $\Gamma_j$ too. 
Therefore the linear subspace $\widetilde{N}:=\overline{\pi_p^{-1}\left(\pi_p(N)\right)}=\Span(N,p)\subset \mathbb{P}^n$ intersects $\Lambda_j$, and $\dim \widetilde{N}=\dim N+1\leq n-k$. 
Moreover, since $\Lambda_j\subset S$ and $N=\widetilde{N}\cap S$, we deduce that $N\cap \Lambda_j\neq \emptyset$.
Thus $L$ meets $\Lambda_j$ and we conclude that $\Lambda_1,\dots,\Lambda_{r}$ are $\SP(n-k)$, as in condition $(B)$.
\end{proof}
\end{lemma}

\begin{lemma}\label{lem:claimSP}
Let $\Lambda_1,\dots,\Lambda_d\subset\mathbb{P}^n$ be $(k-1)$-planes $\SP(n-k)$.
Consider a partition of the set $\{1,\dots,d\}$ consisting of $m\geq 2$ parts of the form
\begin{equation}\label{eq: partition1}
D_1:=\{1,\dots,d_1\}, D_2:=\{d_1+1,\dots,d_1+d_2\},\dots,D_m:=\{\sum_{i=1}^{m-1} d_i+1,\dots,d\},
\end{equation}
and for any $j=1,\dots, m$, let $d_j$ be the cardinality of $D_j$, and let $S_j$ denote the linear span of the $(k-1)$-planes indexed by the part $D_j$.
Suppose that for any $j=1,\dots, m-1$,
\begin{itemize}
\item[(i)] the $(k-1)$-planes indexed by $D_j$ are $\SP(n-k)$;
\item[(ii)] the $(k-1)$-planes indexed by $\{1,\dots,d\}\smallsetminus D_j$ are not $\SP(n-k)$;
\item[(iii)] $\dim S_j\leq d_j+k-3-\varepsilon_j$ for some $\varepsilon_j\geq 0$.
\end{itemize}
Then
\begin{equation}\label{eq: claimSP}
\dim\Span\left(\Lambda_1,\dots,\Lambda_{d}\right)\leq\dim S_m+\sum_{j=1}^{m-1} \left(d_j-\varepsilon_j\right) -m.
\end{equation}
 
\begin{proof}
Consider the union $J$ of $D_m$ and $r$ other parts of \eqref{eq: partition1}, with $0\leq r\leq m-2$.
Then the $(k-1)$-planes indexed by $J$ are not $\SP(n-k)$; otherwise, by assumption (i) and Remark \ref{rem:unionSP}, the union of $J$ and $m-2-r$ parts not contained in $J$ would index a collection of $(k-1)$-planes $\SP(n-k)$, contradicting assumption (ii).
In particular, the $(k-1)$-planes indexed by $D_m$ are not $\SP(n-k)$.

\smallskip
In order to prove \eqref{eq: claimSP}, we argue by induction on $m\geq 2$. 
If $m=2$, partition \eqref{eq: partition1} is given by $D_1=\{1,\dots,d_1\}$ and $D_2=\{d_1+1,\dots,d\}$.
We note that $\Lambda_{1},\dots, \Lambda_{d_1}$ are $\SP(n-k)$ and, in particular, $d_1\geq 2$.
Moreover, $\Lambda_{d_1+1},\dots, \Lambda_{d}$ are not $\SP(n-k)$, so Lemma \ref{lem:intersezioneNonSP} ensures that $\dim(S_1\cap S_2)\geq k-1$.
Since $\dim S_1\leq d_1+k-3-\varepsilon_1$, we conclude that 
$$
\dim\Span\left(\Lambda_1,\dots,\Lambda_{d}\right)=\dim \Span(S_1, S_2) \leq \dim S_1 + \dim S_2 - (k-1) \leq \dim S_2 + d_1-\varepsilon_1-2,
$$
as in \eqref{eq: claimSP}.

\smallskip
Then we assume $m\geq 3$. 
Let $P\subset S_{1}$ be a general subspace of dimension $\alpha:=\dim S_{1}-k+1$. 
We distinguish three cases: 1) $P\cap S_m\neq\emptyset$; 2) $P\cap S_j\neq\emptyset$ for some $j=2,\dots,m-1$ and $P\cap S_m=\emptyset$; 3) $P\cap S_j=\emptyset$ for any $j=2,\dots,m$.
The key point in each case is showing that there exists a $S_j$ (or a linear space $\Sigma$ containing some spaces among $S_2,\dots,S_m$) which intersects $S_1$ along a linear space of large dimension, so that $\dim\Span\left(S_1,S_{j}\right)$ (resp. $\dim\Span\left(S_1,\Sigma\right)$) is not too large.

\smallskip
\underline{Case 1}. Suppose that $P\cap S_m\neq\emptyset$.
It follows by the generality of $P$ that $\dim\left(S_{1}\cap S_m\right)\geq k-1$, and hence
\begin{equation}\label{eq: dim2}
\dim \Span(S_{1},S_m)=\dim S_{1} + \dim S_{m}-  \dim (S_{1}\cap S_m)\leq \dim S_m +d_{1}-\varepsilon_{1}-2.
\end{equation}
Then we consider the partition of $\{1,\dots,d\}$ given by
\begin{equation}\label{eq: partition2}
D_2, D_3, \dots, D_{m-1},D_{1}\cup D_{m},
\end{equation}
which is obtained by joining the first and the last part of \eqref{eq: partition1}.
We note that \eqref{eq: partition2} consists of $m-1$ parts satisfying assumptions (i), (ii) and (iii), and the span of the $(k-1)$-planes indexed by the last part is $\Span(S_{1},S_m)$.
Thus we deduce by induction and \eqref{eq: dim2} that
\begin{align*}
\dim\Span(\Lambda_1,\dots,\Lambda_d)&\leq \dim\Span(S_{1},S_m)+\sum_{j=2}^{m-1} (d_j - \varepsilon_j ) -(m-1)\\
&\leq \dim S_m+\sum_{j=1}^{m-1} (d_j-\varepsilon_j) -m-1,
\end{align*}
so that \eqref{eq: claimSP} holds.

\smallskip
\underline{Case 2}. Analogously, suppose that $P\cap S_j\neq\emptyset$ for some $j=2,\dots,m-1$ and, without loss of generality, set $j=2$.
By arguing as above, the generality of $P$ yields $\dim\left(S_{1}\cap S_{2}\right)\geq k-1$ and 
\begin{equation}\label{eq: dim3}
\begin{aligned}
\dim \Span(S_{1},S_{2})&\leq \left(d_{1}+k-3-\varepsilon_{1}\right)+\left(d_{2}+k-3-\varepsilon_{2}\right)-(k-1)\\
&=d_{1}+d_{2}+k-3-\widetilde{\varepsilon},
\end{aligned}
\end{equation}
where $\widetilde{\varepsilon}:=\varepsilon_{1}+\varepsilon_{2}+2\geq 2$.
Then we consider the partition of $\{1,\dots,d\}$ given by
\begin{equation}\label{eq: partition3}
D_1\cup D_2, D_3, \dots, D_{m},
\end{equation}
which is obtained by joining the first and the second part of \eqref{eq: partition1}.
Then \eqref{eq: partition3} consists of $m-1$ parts, where the first part indexes a collection of $d_{1}+d_{2}$ linear spaces of dimension $(k-1)$, which are $\SP(n-k)$ by Remark \ref{rem:unionSP}, and their linear span is $\Span(S_{1},S_{2})$.
In the light of \eqref{eq: dim3}, assumptions (i), (ii) and (iii) are satisfied, and we deduce by induction that
\begin{align*}
\dim\Span(\Lambda_1,\dots,\Lambda_d)&\leq \dim S_m+(d_{1}+d_{2}-\widetilde{\varepsilon})+\sum_{j=3}^{m-1} (d_j - \varepsilon_j ) -(m-1)\\
&\leq \dim S_m+\sum_{j=1}^{m-1} (d_j-\varepsilon_j) -m-1,
\end{align*}
so that \eqref{eq: claimSP} holds.

\smallskip
\underline{Case 3}. Finally, it remains to prove \eqref{eq: claimSP} when the $\alpha$-plane $P\subset S_1$ is such that $P\cap S_j=\emptyset$ for any $j=2,\ldots,m$, where $\alpha:=\dim S_1-k+1$.
Since $P$ is a subspace of codimension $k-1$ in $S_1$, we deduce that $P\cap \Lambda_i=\emptyset$ if and only if $i=d_1+1,\dots,d$.
Then we consider the projection $\pi_P\colon\mathbb{P}^n \dashrightarrow \mathbb{P}^{n-\alpha-1}$ from $P$, and we set $\Gamma_i:=\pi_P(\Lambda_i)$ for any $i=d_1+1,\dots,d$.  
Denoting $S'_j:=\pi_P(S_j)$ for any $j=2,\ldots,m$, it is easy to see that each $S'_j$ is spanned by the $(k-1)$-planes $\Gamma_i$ indexed by the part $D_j$.
We would like to use induction on $\Gamma_{d_1+1},\dots,\Gamma_d\subset \mathbb{P}^{n-\alpha-1}$ with respect to the partition of $\{d_1+1,\dots,d\}$ given by
\begin{equation}\label{eq: partition5}
D_2, D_3,\dots, D_m,
\end{equation}
which is obtained by removing the first part of \eqref{eq: partition1}.
Thanks to Lemma \ref{lem:sottoSP}, $\Gamma_{d_1+1},\dots,\Gamma_d\subset \mathbb{P}^{n-\alpha-1}$ are $\SP(n-\alpha-1-k)$, and the $(k-1)$-planes $\Gamma_{i}$ indexed by $D_j$ satisfy assumption (i) for any $j=2,\dots,m-1$.
Moreover, $\dim S'_j=\dim S_j$ for any $j=2,\ldots,m$, so that assumption (iii) holds.

\smallskip
\hspace{1cm}\underline{Case 3(a)}. Suppose in addition that $\Gamma_{d_1+1},\dots,\Gamma_d\subset \mathbb{P}^{n-\alpha-1}$ satisfy assumption (ii) with respect to partition \eqref{eq: partition5}. 
We deduce by induction that 
$$
\dim\Span(\Gamma_{d_1+1},\dots,\Gamma_d)\leq \dim S_m+\sum_{j=2}^{m-1} (d_j - \varepsilon_j ) -(m-1).
$$
Let $\Sigma:=\Span\left(\Lambda_{d_1+1},\dots,\Lambda_d\right)$. 
In the light of assumptions (i) and (ii), Lemma \ref{lem:intersezioneNonSP} ensures that $\dim(S_1\cap \Sigma)=k-1+\gamma$ for some $\gamma\geq 0$. 
As $P$ is a general linear space of codimension $k-1$ in $S_1$, it follows that $\dim(P\cap \Sigma)=\gamma$. 
Moreover, $\pi_P(\Sigma)=\Span(\Gamma_{d_1+1},\dots,\Gamma_d)$ and hence
$$
\dim \Sigma=\dim\Span(\Gamma_{d_1+1},\dots,\Gamma_d) +\gamma +1\leq  \dim S_m+\sum_{j=2}^{m-1} (d_j - \varepsilon_j ) -m+2+\gamma.
$$
Therefore $\dim \Span\left(\Lambda_{1},\dots,\Lambda_d\right)=\dim\Span(S_1,\Sigma)=\dim S_1+\dim\Sigma-\dim\left(S_1\cap \Sigma\right)$, so that
\begin{align*} 
\dim \Span\left(\Lambda_{1},\dots,\Lambda_d\right)&\leq (d_1+k-3-\varepsilon_1)+ \left(\dim S_m+\sum_{j=2}^{m-1} (d_j - \varepsilon_j ) -m+2+\gamma\right) -(k-1+\gamma)\\
&=\dim S_m+\sum_{j=1}^{m-1} (d_j - \varepsilon_j )-m,
\end{align*}
which is \eqref{eq: claimSP}.

\smallskip
\hspace{1cm}\underline{Case 3(b)}. Finally, suppose that $\Gamma_{d_1+1},\dots,\Gamma_d$ do not satisfy assumption (ii) with respect to \eqref{eq: partition5}.
Without loss of generality, we assume that $\Gamma_{d_1+d_2+1},\dots,\Gamma_d$ are $\SP(n-\alpha-1-k)$, i.e. assumption (ii) fails for the $(k-1)$-planes $\Gamma_{i}$ such that $i\in\{d_1+1,\dots,d\}\smallsetminus D_2$.

By Lemma \ref{lem:sottosopraSP}, the $(k-1)$-planes $\Gamma_{i}\subset \mathbb{P}^{n-\alpha-1}$ indexed by $D_m$ are not $\SP(n-\alpha-1-k)$, because at the beginning of the proof we pointed out that the corresponding $\Lambda_{i}\subset \mathbb{P}^{n}$ are not $\SP(n-k)$. 
Let $1\leq t\leq m-3$ be an integer such that  the $(k-1)$-planes $\Gamma_{i}$ indexed by the union of $D_m$ and $t$ other parts of \eqref{eq: partition5} satisfy the hypothesis of the lemma. 
The existence of $t$ is granted by the fact that the $(k-1)$-planes $\Gamma_{i}$ such that $i\in D_m$ are not $\SP(n-\alpha-1-k)$, whereas $\Gamma_{d_1+d_2+1},\dots,\Gamma_d$ are $\SP(n-\alpha-1-k)$. 
Without loss of generality, we assume that the last $t+1$ parts of \eqref{eq: partition5}---i.e. $D_{m-t},\dots,D_m$---satisfy the hypotheses of the lemma.

Let $T:=\Span\left(S_{m-t},\dots,S_{m}\right)$ be the span of the $(k-1)$-planes $\Lambda_{i}\subset \mathbb{P}^{n}$ with $i\in  D_{m-t}\cup\dots\cup D_m$.
We note that $P\cap T\neq\emptyset$; otherwise Lemma \ref{lem:sottosopraSP} would assure that these $(k-1)$-planes are $\SP(n-k)$, contradicting the remark at the beginning of the proof.
Let $\delta :=\dim (P\cap T)\geq 0$, so that $\dim(S_1\cap T)=k-1+\delta$.
Since $\dim S_j'=\dim S_j$ for any $j=2,\dots, m$, using induction on the $(k-1)$-planes $\Gamma_{i}\subset \mathbb{P}^{n-1}$ indexed by $D_{m-t}\cup\dots\cup D_m$, we obtain $\dim \pi_P(T)\leq\dim S_m+\sum_{j=m-t}^{m-1} (d_j- \varepsilon_j) -t-1$.
Then we set $\Sigma:=\overline{\pi^{-1}_P\left(\pi_P(T)\right)}$, and we deduce that
\begin{equation}\label{eq: dim6}
\dim \Sigma= \dim \pi_P(T)+\dim P+1 \leq \dim S_m+\sum_{j=m-t}^{m-1} (d_j- \varepsilon_j) +(d_1-\varepsilon_1)-(t+2).
\end{equation}
We point out that $S_1\cap \Sigma$ contains $\Span(P,S_1\cap T)$, whose dimension is 
$$
\dim \Span(P,S_1\cap T)=\dim P+\dim(S_1\cap T)-\dim(P\cap T)
=\dim P+k-1+\delta-\delta=\dim S_1.
$$
Therefore $S_1\subset \Sigma$, so that $\Sigma=\Span\left(S_1,S_{m-t},\dots,S_m\right)$.

To conclude, we consider the partition of $\{1,\dots,d\}$ given by
\begin{equation}\label{eq: partition6}
D_2, D_3, \dots,D_{m-t-1}, D_1\cup D_{m-t}\cup\dots\cup D_m,
\end{equation}
which is obtained from \eqref{eq: partition1} by joining the first part and the last $t+1$ parts.
By assumptions, the $(k-1)$-planes $\Lambda_1,\dots,\Lambda_d$ satisfy the hypotheses of the lemma with respect to partition \eqref{eq: partition6}.
So we conclude by induction and inequality \eqref{eq: dim6} that
\begin{align*}
\dim\Span(\Lambda_1,\dots,\Lambda_d)&\leq \dim \Span\left(S_1,S_{m-t},\dots,S_m\right)+\sum_{j=2}^{m-t-1} (d_j - \varepsilon_j ) -(m-t-1)\\
&\leq \dim S_m+\sum_{j=1}^{m-1} (d_j-\varepsilon_j) -m-1,
\end{align*}
which implies \eqref{eq: claimSP}.
\end{proof}
\end{lemma}

\smallskip
We can finally prove Theorem \ref{thm:dimSP}.

\begin{proof}[Proof of Theorem \ref{thm:dimSP}] 
If $n\leq d+k-3$, the assertion is trivial.
So we assume hereafter that $n\geq d+k-2$.
We argue by induction on $d\geq 2$.

If $d=2$, the $(k-1)$-planes $\Lambda_1,\Lambda_2\subset \mathbb{P}^n$ are $\SP(n-k)$ if and only if they coincide, so  the statement is true.

\smallskip
Then we assume $d\geq 3$.
We point out that there exists $j\in\{1,\dots,d\}$ such that $\Lambda_j\neq\Lambda_i$ for all $i\neq j$; otherwise, by grouping the indexes of coincident $(k-1)$-planes, we would obtain a partition of $\{1,\dots,d\}$, where each part consists of at least 2 coincident $(k-1)$-planes, which are trivially $\SP(n-k)$, but this is impossible as the sequence $\Lambda_1,\dots,\Lambda_d$ is indecomposable. 
Without loss of generality, we assume $j=d$. 

Therefore, given a general point $p\in\Lambda_d$, we have that $p\not\in\Lambda_i$ for any $i=1,\dots,d-1$. 
Let us consider the projection $\pi_p:\mathbb{P}^n\dashrightarrow\mathbb{P}^{n-1}$ from $p$, and for any $i=1,\dots,d-1$, let us set $\Gamma_i:=\pi_p(\Lambda_i)$. 
By Lemma \ref{lem:sottoSP}, the $(k-1)$-planes $\Gamma_1,\dots,\Gamma_{d-1}\subset\mathbb{P}^{n-1}$ are $\SP(n-k-1)$. 

\smallskip
If the sequence $\Gamma_1,\dots,\Gamma_{d-1}$ is indecomposable, we deduce by induction that $\dim\Span\left(\Gamma_1,\dots,\Gamma_{d-1}\right)\leq d+k-4$.
Therefore, the closure of $\pi_p^{-1}\left(\Span\left(\Gamma_1,\dots,\Gamma_{d-1}\right)\right)$ is a linear space $\Sigma\subset \mathbb{P}^n$ of dimension at most $d+k-3$.
Since $\Lambda_1,\dots,\Lambda_{d-1}\subset \Sigma$, Lemma \ref{lem:B1} implies that $\Lambda_{d}\subset \Sigma$ too. 
Thus we conclude that $\dim \Span\left(\Lambda_1,\dots,\Lambda_{d}\right)\leq d+k-3$, as wanted.

\smallskip
So we assume that $\Gamma_1,\dots,\Gamma_{d-1}$ is a decomposable sequence, i.e. there exists a partition of $\{1,\dots,d-1\}$ in $m\geq 2$ parts, where each part $\{i_1,\dots,i_t\}$ is such that the corresponding $(k-1)$-planes $\Gamma_{i_1},\dots,\Gamma_{i_t}$ are $\SP(n-k-1)$.
Let 
\begin{equation}\label{eq: partition7}
D_1:=\{1,\dots,d_1\}, D_2:=\{d_1+1,\dots,d_1+d_2\},\dots, D_m:=\{\sum_{i=1}^{m-1} d_i+1,\dots,d-1\}
\end{equation}
denote such a partition, where the part $D_j$ has cardinality $d_j$ and $\sum_{j=1}^md_j=d-1$.
Furthermore, up to taking a finer partition, we may assume that each part indexes an indecomposable sequence of $(k-1)$-planes $\Gamma_i$.
We point out that the partition does not depend on the choice of the general point $p\in \Lambda_d$, because $\Lambda_d$ is irreducible and the number of partitions of $\{1,\dots,d-1\}$ is finite. 
For any $j=1,\dots, m$, let 
$$
S'_j:=\Span\left(\Gamma_i\left|i\in D_j\right.\right)\subset \mathbb{P}^{n-1} \quad\text{and}\quad S_j:=\Span\left(\Lambda_i\left|i\in D_j\right.\right)\subset \mathbb{P}^{n}.
$$
By induction, $\dim S'_j\leq d_j+k-3$ for any $j=1,\dots,m$.

Since we assumed $n\geq d+k-2$, we have that $2\leq d_j<d-1\leq n-k+1$ for any $j=1,\dots,m$.
Hence Lemma \ref{lem:sopraSP} ensures that for any $j=1,\dots,m$, either $\Lambda_d\subset S_j$---as in case (A)---or the $(k-1)$-planes $\Lambda_i$ indexed by $D_j$ are $\SP(n-k)$, as in case (B).
Let $0\leq t\leq m$ be the largest number of parts such that $\Lambda_d\subset S_j$ and, without affecting generality, assume that this happens for $j=m-t+1,\dots,m$. 
Therefore
\begin{equation}\label{eq:dim7}
\dim\Span(S_{m-t+1},\dots,S_m)\leq \sum_{j=m-t+1}^m(d_j+k-3)-(t-1)(k-1)=\sum_{j=m-t+1}^m d_j -2t+k-1.
\end{equation}
 If $t=m$, that is $\Lambda_d\subset S_j$ for any $j=1,\dots,m$, the previous formula gives
$$
\dim\Span(\Lambda_{1},\dots,\Lambda_d)=\dim\Span(S_{1},\dots,S_m)\leq \sum_{j=1}^m d_j -2m+k-1,
$$
so the assertion follows as $ \sum_{j=1}^m d_j=d-1$ and $m\geq 2$.

If instead $0\leq t<m$, the first $m-t$ parts of \eqref{eq: partition7} satisfy condition (B) of Lemma \ref{lem:sopraSP}, and they do not satisfy condition (A), i.e. for any $j=1,\dots,m-t$, the $(k-1)$-planes $\Lambda_i$ indexed by $D_j$ are $\SP(n-k)$ and $\Lambda_d\not\subset S_j$.
In particular, $p\not\in S_j$ because $p\in \Lambda_d$ is a general point.
Therefore the $(k-1)$-planes indexed by $D_j$ give an indecomposable sequence, otherwise by Lemma \ref{lem:sottosopraSP} the corresponding sequence of $\Gamma_i=\pi_p(\Lambda_i)$ would be decomposable, but this is not the case.
Then we deduce by induction that $\dim S_j\leq d_j+k-3$ for any $j=1,\ldots, m-t$.

Finally, let us consider the partition of $\{1,\dots,d\}$ in $m-t+1$ parts given by
\begin{equation}\label{eq: partition8}
D_1, D_2,\dots,D_{m-t}, D_{m-t+1}\cup \dots\cup D_m\cup\{d\},
\end{equation}
where the first $m-t$ parts are those of \eqref{eq: partition7}, whereas the last part is the union of $\{d\}$ with the last $t$ parts of \eqref{eq: partition7}.
By summing up, the sequence $\Lambda_1,\dots,\Lambda_d$ is indecomposable, and for any $j=1,\dots,m-t$, the part $D_j$ indexes a sequence of $(k-1)$-planes which are $\SP(n-k)$, whose span is $S_j$ and satisfies $\dim S_j\leq d_j+k-3$.
Therefore the assumptions of Lemma \ref{lem:claimSP} holds with respect to partition \eqref{eq: partition8}.

Concerning the last part of \eqref{eq: partition8}, let $B$ be the span of the corresponding $(k-1)$-planes.
If $t=0$, the last part is just $\{d\}$ and $\dim B=\dim \Lambda_d =k-1$.
If instead $t>0$, then $B=\Span(S_{m-t+1},\dots,S_m)$, because $\Lambda_d\subset S_j$ for any $j=m-t+1,\dots,d$, and $\dim B$ is bounded by \eqref{eq:dim7}.
Thus for any $t\geq 0$, we obtain from Lemma \ref{lem:claimSP} that
$$
\dim\Span(\Lambda_1,\dots,\Lambda_d)\leq \dim B+ \sum_{j=1}^{m-t}d_j-(m-t+1)
\leq \sum_{j=1}^m d_j +k-m-t-2.
$$
In particular, the assertion holds since $ \sum_{j=1}^m d_j=d-1$, $m\geq 2$ and $t\geq 0$. 
\end{proof}

\smallskip
\subsection{Further remarks}\label{sub:remarks}
To conclude this section, we collect some remarks and possible developments concerning linear subspaces of $\mathbb{P}^n$ in special position and, more generally, points of $\mathbb{G}(k-1,n)$ satisfying the Cayley--Bacharach condition with respect to $\left|\mathcal{O}_{\mathbb{G}(k-1,n)}(r)\right|$, for some $r\geq 1$.

\begin{remark}\label{rem:GK}
In this section, we discussed the dimension of the span of $(k-1)$-dimensional linear spaces in $\mathbb{P}^n$, as being $\SP(n-k)$ is a property descending from the Cayley--Bacharach condition with respect to $\left|\mathcal{O}_{\mathbb{G}(k-1,n)}(1)\right|$.

More generally, it would be interesting to investigate the same problem for $(k-1)$-dimensional linear spaces in $\mathbb{P}^n$ parameterized by a set $\Gamma\subset \mathbb{G}(k-1,n)$ of points, which satisfies the Cayley--Bacharach condition with respect to $\left|\mathcal{O}_{\mathbb{G}(k-1,n)}(r)\right|$ for larger values of $r\geq 1$.
Moreover, this may lead to some progress in the study of measures of irrationality of Fano schemes of complete intersections.  

We recall that the \emph{Fano scheme} of a projective variety $X\subset \mathbb{P}^n$ is the scheme $F_k(X)\subset \mathbb{G}(k,n)$ parameterizing $k$-dimensional linear spaces contained in $X$.
In \cite{GK}, the authors study the geometry of lines in special position in $\mathbb{P}^4$, in order to determine the covering gonality, the connecting gonality and the degree of irrationality of the Fano surface $F_1(X)$ of a smooth cubic threefold $X\subset \mathbb{P}^4$.
In this setting, the connection between measures of irrationality and lines in special position depends on the following fact. 
By the same techniques we use in the next section, the points of $F_1(X)$ computing its covering gonality (resp. the connecting gonality or the degree of irrationality) satisfy the Cayley--Bacharach condition with respect to the canonical linear series of $F_1(X)$, which is the linear series $\left|\mathcal{O}_{F_1(X)}(1)\right|$ cut out on $F_1(X)$ by hyperplanes of $\mathbb{P}^{9}$ under the Pl\"ucker embedding of $\mathbb{G}(1,4)$.  

The same techniques apply in general to Fano schemes $F_k(X)\subset \mathbb{G}(k,n)$ of smooth complete intersections $X\subset \mathbb{P}^n$ of large degree, whose canonical bundle has the form $\mathcal{O}_{F_k(X)}(r)$, with $r\geq 1$ (see e.g. \cite[Remarques 3.2]{DM}), hence motivating the study of the geometry of $k$-planes satisfying the Cayley--Bacharach condition with respect to $\left|\mathcal{O}_{\mathbb{G}(k,n)}(r)\right|$.
\end{remark}

\begin{remark}\label{rem:LU}
According to \cite{LU}, we say that the union $\mathcal{P}=P_1\cup \dots \cup P_m\subset \mathbb{P}^n$ of positive-dimensional linear spaces is a \emph{plane configuration of length} $m$ \emph{and dimension} $\dim \mathcal{P}:=\sum_{i=1}^m\dim P_i$. 
Given two positive integers $d$ and $r$, the authors study the least dimension of a plane configuration containing a given set $\Gamma\subset \mathbb{P}^n$ of $d$ points satisfying the Cayley--Bacharach condition with respect to $|\mathcal{O}_{\mathbb{P}^n}(r)|$.

We note that Corollary \ref{cor:dimSP} can be rephrased in terms of plane configurations.
Namely, let $\Lambda_1,\dots,\Lambda_d\subset \mathbb{P}^{n}$ be linear spaces of dimension $(k-1)$ as in Corollary \ref{cor:dimSP}.
By arguing as in its proof, we easily see that $\Lambda_1,\dots,\Lambda_d$ are contained in a plane configuration $\mathcal{P}=P_1\cup \dots \cup P_m\subset \mathbb{P}^n$ of dimension $\dim \mathcal{P}=d+ m(k-3)$, where each $P_i$ is the linear span of the $(k-1)$-planes indexed by the $i$-th part.

So, in analogy with \cite{LU}, it would be interesting to determine the least dimension of a plane configuration containing the $(k-1)$-planes $\Lambda_1,\dots,\Lambda_d\subset \mathbb{P}^{n}$ parameterized by a set $\Gamma\subset \mathbb{G}(k-1,n)$ of $d$ points, which satisfies the Cayley--Bacharach condition with respect to $\left|\mathcal{O}_{\mathbb{G}(k-1,n)}(r)\right|$ for some $r\geq 1$.
\end{remark}

\begin{remark}\label{rem:LP}
It follows from \cite[Lemma 2.5]{LP} and \cite[Theorem A]{Pic} that, if $\Gamma\subset \mathbb{P}^{n}$ is a finite set of points satisfying the Cayley--Bacharach condition with respect to $|\mathcal{O}_{\mathbb{P}^n}(r)|$ and $|\Gamma|\leq h(r-3+h)-1$ for some $r\geq 1$ and $2\leq h\leq 5$, then $\Gamma$ lies on a curve of degree $h-1$ (actually, this is true for all $h\geq2$ when $n=2$).
Moreover, questions about analogous results are discussed in \cite[Section 7.3]{LU}.

Similarly, given a finite set $\Gamma\subset \mathbb{G}(k-1,n)$ of points satisfying the Cayley--Bacharach condition with respect to $\left|\mathcal{O}_{\mathbb{G}(k-1,n)}(r)\right|$ and having small cardinality, it would be interesting to understand whether the $(k-1)$-planes parameterized by $\Gamma$ lie on a $k$-dimensional variety of low degree.
For instance, looking at \cite[Examples 3.5 and 3.6]{B1} and \cite[Proposition 5.2]{GK}, one deduces that $3$ lines satisfying $\SP(n-2)$ lie on a plane and $4$ lines satisfying $\SP(n-2)$ lie on a---possibly reducible---quadric surface. 
\end{remark}

%%%%%%%%%%%%%%%%%%%%%%%%%%%%%%%%%%%%%%%
%%%%%%%%%%%%%%%%%%%%%%%%%%%%%%%%%%%%%%%
%%%%%%%%%%%%%%%%%%%%%%%%%%%%%%%%%%%%%%%

%\medskip
\section{Covering gonality of $C^{(3)}$ and $C^{(4)}$}\label{section:covgonCk} 

In this section, we study the covering gonality of symmetric products of curves and, in particular, we are aimed at proving Theorem \ref{thm:covgon}.
\smallskip
\subsection{Preliminaries}

To start, we recall some preliminary facts concerning covering gonality of projective varieties and linear series on smooth curves.

Let $X$ be an irreducible complex projective variety of dimension $n$.

\begin{definition}\label{def:covfam}
A \emph{covering family of} $d$\emph{-gonal curves} on $X$ consists of a smooth family $\mathcal{E}\stackrel{\pi}{\longrightarrow} T$ of irreducible curves endowed with a dominant morphism $f\colon \mathcal{E}\longrightarrow X$ such that for general $t\in T$, the fiber  $E_t:=\pi^{-1}(t)$ is a smooth curve with gonality $\gon(E_t)=d$ and the restriction $f_{t}\colon E_t\longrightarrow X$ of $f$ is birational onto its image. 
\end{definition}

It is worth noticing that the covering gonality of $X$ coincides with the least integer $d>0$ such that a covering family of $d$-gonal curves exists (see e.g. \cite[Lemma 2.1]{GK}).

Given a covering family $\mathcal{E}\stackrel{\pi}{\longrightarrow} T$ of $d$-gonal curves, both the varieties $T$ and $\mathcal{E}$ can be assumed to be smooth, with $\dim(T)=n-1$ (see \cite[Remark 1.5]{BDELU}).
Moreover, up to base changing $T$, there is a commutative diagram
\begin{equation}\label{diagram:covgon}
\xymatrix{X  & \mathcal{E} \ar[l]_f \ar[dr]_-{\pi} \ar[r]^-\varphi & T\times \mathbb{P}^1 \ar[d]^-{\mathrm{pr_1}} \\ & & T,\\}
\end{equation}
where the restriction $\varphi_t\colon E_t \longrightarrow \{t\}\times\mathbb{P}^1  \cong \mathbb{P}^1$ is a $d$-gonal map (cf. \cite[Example 4.7]{B1}). 

\smallskip
Turning to symmetric products of curves, let $C$ be a smooth projective curve of genus $g\geq 2$ and let $C^{(k)}$ be its $k$-fold symmetric product, with $2\leq k\leq g-1$.
Let ${\phi\colon C\longrightarrow \mathbb{P}^{g-1}}$ be the canonical map and let $\mathbb{G}(k-1,g-1)$ denote the Grassmann variety of ${(k-1)}$-planes in
$\mathbb{P}^{g-1}$.
Thanks to the General Position Theorem (see \cite[p.\,109]{ACGH}), we can define the \emph{Gauss map} 
$$\gamma\colon C^{(k)}\dashrightarrow \mathbb{G}(k-1,g-1),$$ 
which sends a general point ${p_1+\dots+p_k\in C^{(k)}}$ to the point of $\mathbb{G}(k-1,g-1)$ parameterizing the $(k-1)$-plane $\Span\left(\phi(p_1),\dots,\phi(p_k)\right)\subset \mathbb{P}^{g-1}$. 
We would also like to recall that $C^{(k)}$ is a variety of general type, where $H^0\left(C^{(k)},\omega_{C^{(k)}}\right)\cong \bigwedge^k H^0\left(C,\omega_{C}\right)$ and the canonical map factors through the Gauss map and the Pl\"ucker embedding $p\colon \mathbb{G}(k-1,g-1)\longrightarrow \mathbb{P}^{{g \choose k }-1}$, that is
\begin{equation}\label{diagram:canonical map}
\xymatrix{C^{(k)} \ar@{-->}[drr]_-{\gamma} \ar[rr]^-{|\omega_{C^{(k)}}|} & & \bP(\bigwedge^kH^{0}(C,\omega_C))\cong \bP^{{g \choose k }-1} \\  & &\mathbb{G}(k-1,g-1) \ar[u]^-{p} \\}
\end{equation}  
(cf. \cite{Mac}).
      
The following theorem is the key result connecting covering families of irreducible $d$-gonal curves on $C^{(k)}$ and $(k-1)$-planes of $\mathbb{P}^{g-1}$ in special position with respect to $(g-1-k)$-planes. 
It is somehow included in \cite[Section 4]{B1}, but we present a proof for the sake of completeness.

\begin{theorem}\label{thm:covfamCk}
Let $C$ be a smooth projective curve of genus $g\geq 2$ and let $C^{(k)}$ be its $k$-fold symmetric product, with $2\leq k\leq g-1$.
Let $\mathcal{E}\stackrel{\pi}{\longrightarrow} T$ be a covering family of $d$-gonal curves on $C^{(k)}$.
Using notation as above, consider a general point $(t,y)\in T\times \mathbb{P}^1$ and let $\varphi^{-1}(t,y)=\left\{x_1,\ldots,x_d\right\}\subset E_t$ be its fiber.\\
Then the $(k-1)$-planes parameterized by $(\gamma\circ f) (x_1),\dots,(\gamma\circ f) (x_d)\in \mathbb{G}(k-1,g-1)$ are in special position with respect to $(g-1-k)$-planes of $\mathbb{P}^{g-1}$.

\begin{proof}
According to \cite[Definition 4.1]{B1} and \cite[Example 4.7]{B1}, the covering family $\mathcal{E}\stackrel{\pi}{\longrightarrow} T$ induces a correspondence with null trace on $(T\times \mathbb{P}^1)\times C^{(k)}$ given by
$$
\Gamma:=\overline{\left\{\left.\big((t,y),P\big)\in (T\times \mathbb{P}^1)\times C^{(k)}\right|P\in f(E_t) \text{ and } f_t^{-1}(P)\in\varphi^{-1}(t,y) \right\}}
$$
(i.e. the general point $\big((t,y),P\big)\in \Gamma$ is such that $P\in f(E_t)$, and the preimage of $P$ on $E_t$ maps to $(t,y)$ under the $d$-gonal map $\varphi_t\colon E_t \longrightarrow \{t\}\times\mathbb{P}^1  \cong \mathbb{P}^1$).
In particular, the projection $\pi_1\colon \Gamma \longrightarrow T\times \mathbb{P}^1$ has degree $d$, since the general fiber $\pi_1^{-1}(t,y)$ consists of $d$ points $\big((t,y),P_i\big)$, where $P_i:=f_t(x_i)$ and $\varphi^{-1}(t,y)=\left\{x_1,\ldots,x_d\right\}\subset E_t$.

Therefore \cite[Proposition 4.2]{B1} ensures that for any $i=1,\dots,d$ and for any canonical divisor $K\in \left|\omega_{C^{(k)}}\right|$ containing $P_1,\dots,\widehat{P_i},\dots,P_d$, we have $P_i\in K$.
In addition, it follows from \cite[Lemma 2.1]{B1} that any $(g-1-k)$-plane $L\subset \mathbb{P}^{g-1}$ defines a canonical divisor $K_L$ on $C^{(k)}$, and $P\in K_L$ if and only if $\gamma(P)\in \mathbb{G}(k-1,g-1)$ is a point of the Schubert cycle $\sigma_1(L):=\left\{\left.[\Lambda]\in\mathbb{G}(k-1,g-1) \right|\Lambda\cap L\neq\emptyset\right\}$, which parameterizes $(k-1)$-planes of $ \mathbb{P}^{g-1}$ intersecting $L$.
Thus, for any  $i=1,\dots,d$ and for any $(g-1-k)$-plane $L\subset \mathbb{P}^{g-1}$ such that $(\gamma\circ f) (x_1),\dots,\widehat{(\gamma\circ f) (x_i)},\dots,(\gamma\circ f) (x_d)\in \sigma_1(L)$, we have that $(\gamma\circ f) (x_i)\in \sigma_1(L)$, i.e. the $(k-1)$-planes parameterized by $(\gamma\circ f) (x_1),\dots,(\gamma\circ f) (x_d)\in \mathbb{G}(k-1,g-1)$ are in special position with respect to $(g-1-k)$-planes of $\mathbb{P}^{g-1}$.
\end{proof}
\end{theorem}

\begin{remark}\label{rem:upper bound}
We note that for any smooth curve $C$ of genus $g\geq 0$, its $k$-fold symmetric product is covered by the family of curves $\mathcal{C}\stackrel{\pi}{\longrightarrow} C^{(k-1)}$, where the fiber over $P=p_1+\dots+p_{k-1}\in C^{(k-1)}$ is the curve $C_P:=\left\{\left.p_1+\dots+p_{k-1}+q\in C^{(k)}\right|q\in C\right\}$ isomorphic to $C$.
Thus
\begin{equation}\label{eq:upper bound}
\covgon(C^{(k)})\leq \gon(C).
\end{equation} 
\end{remark}

\smallskip
Given two positive integers $r$ and $d$, let $W^r_d$ be the subvariety of $\Pic^d(C)$, which parameterizes complete linear series on $C$ having degree $d$ and dimension at least $r$ (cf. \cite[p.\,153]{ACGH}). 
Moreover, we recall that a $\mathfrak{g}^r_d$ on $C$ is a (possibly non complete) linear series of degree $d$ and dimension $r$.
Finally, before proving Theorem \ref{thm:covgon}, we point out two elementary facts involved in its proof.

\begin{remark}\label{rem:g^r_d}
Let $C\subset \mathbb{P}^{g-1}$ be the canonical model of a smooth non-hyperelliptic curve and let $D=q_1+\dots+q_s\in \Div(C)$ be a reduced divisor. 
It follows from the geometric version of the Riemann--Roch theorem (cf. \cite[p.\,12]{ACGH}) that the complete linear series $|D|$ is a $\mathfrak{g}^r_s$ with $r= s-1-\dim \Span(q_1,\dots,q_s)$.
If in addition $r\geq 1$, then 
\begin{equation}\label{eq:g^r_d}
\dim \Span(q_1,\dots,q_s)\geq\gon(C)-2.
\end{equation}
Indeed, $\dim|D-(r-1)p|\geq \dim |D|-(r-1)=1$ for any $p\in C$, and hence $\deg(D-(r-1)p)=s-r+1\geq \gon(C)$. Since $r= s-1-\dim  \Span(q_1,\dots,q_s)$, we obtain \eqref{eq:g^r_d}.
\end{remark}

\begin{lemma}\label{lem:very ample}
Let $C$ be a smooth non-hyperelliptic curve of genus $g\geq 3$ and, for some effective divisor $D\in \Div(C)$, let $|D|$ be a $\mathfrak{g}^r_d$ such that $r\geq 2$ and $\dim|D-R|= 0$ for any $R\in C^{(r)}$.
Then $|D|$ is a very ample $\mathfrak{g}^{2}_{d}$ on $C$.
\begin{proof}
We note that $d\geq 2r\geq 4$ by Clifford's theorem (see e.g. \cite[p.\,107]{ACGH}).
For any $R\in C^{(r)}$, we have that $\dim|D-R|= \dim|D|-r= 0$.
Since $\dim|D-p|\geq \dim |D|-1$ for any $p\in C$, it follows that $\dim|D-p-q|=\dim |D|-2$ for any $p,q\in C$, so that $|D|$ is very ample (cf. \cite[Proposition V.4.20]{Mir}).

By contradiction, suppose that $r\geq 3$.
For any $B\in C^{(b)}$ having $1\leq b\leq r-2$, the divisor $D-B\in \Div(C)$ is such that $\dim |D-B|=r-b\geq 2$ and $\dim |D-B-A|=0$ for any $A\in C^{(r-b)}$.
Thus we may argue as above, and we deduce that $|D-B|$ is a very ample $\mathfrak{g}^{r-b}_{d-b}$.
In particular, $C$ admits a very ample $\mathfrak{g}^2_e$ with $e:=d-r+2\geq 4$, so that $g=\frac{(e-1)(e-2)}{2}$. 
Analogously, $C$ is endowed with a very ample $\mathfrak{g}^3_{e+1}$, and Castelnuovo's bound (cf. \cite[p.\,107]{ACGH}) implies that $g\leq m(m-1)+m\varepsilon$, where $m=\frac{e-\varepsilon}{2}$ and $\varepsilon \in \{0,1\}$. 
Thus we obtain
\begin{equation}\label{eq:very ample}
\frac{(e-1)(e-2)}{2}=g\leq m(m-1+\varepsilon)= \frac{e-\varepsilon}{2}\cdot \frac{e+\varepsilon-2}{2}.
\end{equation}
Hence we get a contradiction, as inequality \eqref{eq:very ample} fails for any $\varepsilon \in \{0,1\}$ and $e\geq 4$.
\end{proof}
\end{lemma}

\smallskip
\subsection{Proof of Theorem \ref{thm:covgon}}\label{sub:covgon}
In this subsection we prove Theorem \ref{thm:covgon}, and we briefly discuss the covering gonality of $C^{(k)}$ when $0\leq g\leq k$.

\begin{proof}[Proof of Theorem \ref{thm:covgon}]
Let $k\in \{3,4\}$ and let $C$ be a smooth curve of genus $g\geq k+1$, with $\big(g,\gon(C)\big)\neq(k+1,k)$.
In the light of \eqref{eq:upper bound}, we need to prove that $\covgon\big(C^{(k)}\big)\geq \gon(C)$.

Since $C^{(k)}$ is a variety of general type, it is not covered by rational curves. 
It follows that $\covgon\big(C^{(k)}\big)\geq 2$, and the assertion holds when $C$ is a hyperelliptic curve.

\smallskip
We assume hereafter that $C$ is non-hyperelliptic and, in order to simplify notation, we identify the curve with its canonical model $\phi(C)\subset \mathbb{P}^{g-1}$.

Aiming for a contradiction, we suppose that there exists a covering family $\mathcal{E}\stackrel{\pi}{\longrightarrow} T$ of $d$-gonal curves, with $d<\gon(C)$.  
Therefore, as $\gon(C)\leq \left\lfloor\frac{g+3}{2}\right\rfloor$ (see e.g. \cite[Theorem V.1.1]{ACGH}), we have 
\begin{equation}\label{eq:g}
d\leq \gon(C)-1\leq \left\lfloor\frac{g+1}{2}\right\rfloor \quad\text{and}\quad g\geq 2d-1.
\end{equation}
Using notation as in Definition \ref{def:covfam} and \eqref{diagram:covgon}, we consider a general point $(t,y)\in T\times \mathbb{P}^1$ and its fiber $\varphi_t^{-1}(t,y)=\left\{x_1,\ldots,x_d\right\}\subset E_t$ via the $d$-gonal map $\varphi_t\colon E_t \longrightarrow \{t\}\times\mathbb{P}^1$.  
Moreover, for any $i=1,\dots,d$, we define $P_i:=f_t(x_i)\in C^{(k)}$ and we set
\begin{equation}\label{eq:P_i}
P_1=p_1+\dots+p_{k} \quad P_2=p_{k+1}+\dots+p_{2k} \quad \dots\, \quad P_d=p_{(d-1)k+1}+\dots+p_{dk}.
\end{equation}
We point out that, since $(t,y)\in T\times \mathbb{P}^1$ is general, the points $P_i\in C^{(k)}$ are distinct and they lie outside the diagonal of $C^{(k)}$.
Moreover, for $i=1,\dots,d$, the linear space
\begin{equation}\label{eq:Lambda_i}
\Lambda_i:=\Span\left(p_{(i-1)k+1},\dots,p_{ik}\right)\subset \mathbb{P}^{g-1}
\end{equation}
spanned by $\Supp (P_i)$ has dimension $k-1$ and it is parameterized by $\gamma(P_i)\in \mathbb{G}(k-1,g-1)$.
By Theorem \ref{thm:covfamCk}, the $(k-1)$-planes $\Lambda_1,\dots,\Lambda_d$ are in special position with respect to $(g-1-k)$-planes.
In particular, the dimension of the linear span of the points $p_1,\dots,p_{dk}$ is governed by Corollary \ref{cor:dimSP}.

Let us consider the effective divisor $D=D_{(t,y)}\in\Div(C)$ given by
\begin{equation}\label{eq:divisorD}
D:=p_1+\dots +p_{dk}=\sum_{j=1}^h n_jq_j,
\end{equation}
where the points $q_j\in \{p_1,\dots,p_{dk}\}$ are assumed to be distinct and $n_j=\mult_{q_j}(D)$.
We now distinguish some cases depending on the values of the integers $n_j$.

\medskip
\underline{Case A}: suppose that $n_j=1$ for any $j=1,\dots,h$, that is the points $p_1,\dots,p_{dk}$ are all distinct.

\smallskip
\hspace{1cm}\underline{Case A.1}: assume in addition that $d=2$, i.e. the map $\varphi_t\colon E_t \longrightarrow \{t\}\times\mathbb{P}^1$ is hyperelliptic. 
Since the $(k-1)$-planes $\Lambda_1$ and $\Lambda_2$ are $\SP(g-1-k)$, they coincide, that is $\gamma(P_1)=\gamma(P_2)$ (cf. Remark \ref{rem:unionSP}).
In particular, $\dim\Span(p_1,\dots,p_{2k})=k-1$ and the geometric version of the Riemann--Roch theorem (see \cite[p.\,12]{ACGH}) yields
\begin{equation}\label{eq:GRR1}
\dim |D|=\deg D-1-\dim\Span(p_1,\dots,p_{2k})=2k-1-(k-1)=k=\frac{\deg D}{2}.
\end{equation}
Since $D\neq 0$ and $C$ is assumed to be non-hyperelliptic, Clifford's theorem (cf. \cite[p.\,107]{ACGH}) assures that $D$ is a canonical divisor of $C$, and hence $g=k+1$.
We note that for $k=3$, we obtain $g=4$ and $2=d<\gon(C)\leq 3$, but the case $\big(g,\gon(C)\big)= (4,3)$ is excluded by assumption.
Thus $k=4$, and let $\iota\colon C^{(4)}\longrightarrow C^{(4)}$ be the involution $p_1+\dots+p_4 \longmapsto K_{C^{(4)}}-(p_1+\dots+p_4)$ sending a point to the residual of the canonical system. 
We point out that the restriction of $\iota$ to the curve $f(E_t)\subset C^{(4)}$ is the automorphism induced by the hyperelliptic involution of $E_t$.
Therefore the quotient variety $C^{(4)}/\langle \iota \rangle$ is covered by a family of rational curves.
On the other hand, since $D=P_1+P_2$ is a canonical divisor on $C$, we have that $P_2=\iota(P_1)$ with $\gamma(P_1)=\gamma(P_2)$.
Hence the Gauss map factors through $C^{(4)}/\langle \iota \rangle$ and, in the light of diagram \eqref{diagram:canonical map}, also the canonical map of $C^{(k)}$ does.
Thus $C^{(4)}/\langle \iota \rangle$ is a variety of general type, because $C^{(4)}$ is.
So we get a contradiction as $C^{(4)}/\langle \iota \rangle$ can not be covered by rational curves.

\smallskip
\hspace{1cm}\underline{Case A.2}: assume that $d\geq 3$.
Consider the sequence $\Lambda_1,\dots,\Lambda_d\subset \mathbb{P}^{g-1}$ of $(k-1)$-planes satisfying $\SP(g-1-k)$, together with a partition of the set of indexes $\left\{1,\dots,d\right\}$ such that each part $\{i_1,\dots,i_t\}$ corresponds to an indecomposable sequence $\Lambda_{i_1},\dots,\Lambda_{i_t}$ of $(k-1)$-planes that satisfy $\SP(g-1-k)$, as in the assumption of Corollary \ref{cor:dimSP}.

We point out that each part consists of at least $3$ elements. 
Indeed, if we had two planes $\Lambda_{i_1}$ and $\Lambda_{i_2}$ in special position, we could argue as in Case A.1 and we would deduce that $g=k+1$.
However, this is impossible, because we know that $3\leq d\leq\gon(C)-1\leq \left\lfloor\frac{g+1}{2}\right\rfloor$, with $k\in\{3,4\}$ and $\big(k,g,\gon(C)\big)\neq (4,5,4)$ by assumption.

Let $m$ be the number of parts. 
Hence $m\leq \left\lfloor\frac{d}{3}\right\rfloor$ and Corollary \ref{cor:dimSP} implies
\begin{align*}
\dim\Span(p_1,\dots,p_{dk}) & =\dim\Span\left(\Lambda_1,\dots,\Lambda_d\right)\leq d+k-3 +(m-1)(k-2) \notag \\
 &\leq d+k-3+\left(\left\lfloor\frac{d}{3}\right\rfloor-1\right)(k-2).
\end{align*}
By the geometric version of the Riemann--Roch theorem we obtain
\begin{align}\label{eq:GRR2}
\dim |D|&=\deg D-1-\dim\Span(p_1,\dots,p_{dk})\geq dk-1-\left(d+k-3+\left(\left\lfloor\frac{d}{3}\right\rfloor-1\right)(k-2)\right) \notag\\
& = (k-1)d -(k-2)\left\lfloor\frac{d}{3}\right\rfloor>\frac{kd}{2}=\frac{\deg D}{2},
\end{align}
where the last inequality holds because $k\geq 3$.
Hence Clifford's theorem assures that $\deg D\geq 2g$, so that $\dim|D|=\deg D-g=kd-g$ by the Riemann--Roch theorem.
Combining this fact and \eqref{eq:GRR2}, we obtain $kd-g\geq (k-1)d -(k-2)\frac{d}{3}$, which gives $3g\leq d(k+1)$. 
Then we conclude by \eqref{eq:g} that $(k,g,d)=(4,5,3)$ and, in particular, $\gon(C)=4$.
However, $\big(k,g,\gon(C)\big)\neq(4,5,4)$ by assumption, so we get a contradiction.

\medskip
\underline{Case B}: suppose that the points $p_1,\dots,p_{kd}$ are not distinct, i.e. the integers $n_j$ are not all equal to $1$.
Let us denote by 
\begin{equation}\label{eq:N_alpha}
N_{\alpha}:=\left\{\left. q_j\in\Supp(D)\right|n_j=\alpha\right\} 
\end{equation}
the set of the points supporting $D$ such that $\mult_{q_j}D=\alpha$. 
Given some $\alpha\geq 1$ such that $N_{\alpha}\neq \emptyset$, we may assume $N_{\alpha}=\{q_1,\dots,q_s\}$, with $s\leq \left\lfloor\frac{\deg D}{\alpha}\right\rfloor=\left\lfloor\frac{kd}{\alpha}\right\rfloor$.

Since $T\times \mathbb{P}^1$ is irreducible, there exists a suitable open subset $U\subset T\times \mathbb{P}^1$ such that, as we vary $(t,y)\in U$, the multiplicities $n_j$ of the points supporting the divisor $D=D_{(t,y)}$ given by \eqref{eq:divisorD} do not vary, i.e. the cardinality of each $N_{\alpha}$ is constant on $U$. 
Thus we can define a rational map $\xi_{\alpha}\colon T\times \mathbb{P}^1\dashrightarrow C^{(s)}$ sending a general point $(t,y)\in U$ to the effective divisor $Q:=q_1+\dots+q_s\in C^{(s)}$ such that $N_{\alpha}=\{q_1,\dots,q_s\}$. 

For a general $t\in T$, let $\xi_{\alpha,t}\colon \{t\}\times \mathbb{P}^1\dashrightarrow C^{(s)}$ be the restriction of $\xi_\alpha$ to $\{t\}\times \mathbb{P}^1$, and let us consider the composition with the Abel--Jacobi map $u$,
\begin{equation*}
\xymatrix{\{t\}\times\bP^1  \ar@{-->}[r]^-{\xi_{\alpha,t}} & C^{(s)} \ar[r]^-{u} & J(C).\\}
\end{equation*}
As the Jacobian variety $J(C)$ does not contain rational curves, the map $u\circ \xi_{\alpha,t}$ must be constant.

\smallskip
\hspace{1cm}\underline{Case B.1}: suppose that for any $\alpha$ such that $N_{\alpha}\neq\emptyset$, the map $\xi_{\alpha,t}\colon\{t\}\times \mathbb{P}^1\dashrightarrow C^{(s)}$ is non-constant.
It follows from Abel's theorem that $|q_1+\dots+q_s|$ is a $\mathfrak{g}^r_s$  containing $\xi_{\alpha,t}\left(\{t\}\times\bP^1\right)$, so that $r\geq 1$ and $s=|N_\alpha|\geq \gon(C)$.
Hence $\alpha\leq k-1$, because otherwise we would have $s\leq \left\lfloor\frac{kd}{\alpha}\right\rfloor\leq d<\gon(C)$.

We also point out that $\alpha |N_\alpha|=\alpha s$ must be a multiple of $d$. 
To see this fact, let $i=1,\dots,d$ and consider the points $P_i=p_{(i-1)k+1}+\dots+p_{ik}$ in \eqref{eq:P_i}. 
Then the number of points $p_j\in N_\alpha\cap \Supp(P_i)$ does not depend on $i$; otherwise, by varying $(t,y)\in \{t\}\times \mathbb{P}^1$, the points $P_i$ would describe different irreducible components of $E_t$, but such a curve is irreducible. 
Hence $\alpha|N_\alpha|$---which is the number of points among $p_1,\dots,p_{dk}$ contained in $N_\alpha$---must be a multiple of $d$.

We claim that $N_1=\emptyset$ and there exists a unique $2\leq \alpha\leq k-1$ such that $N_{\alpha}\neq\emptyset$, with in particular $|N_\alpha|=\frac{kd}{\alpha}$.
Indeed, if $k=3$, we have that $3d=\deg(D)=|N_1|+2|N_2|$.  
Then there exists $j\in\{1,2\}$ such that $|N_j|\leq d$, but we observed that if $N_j\neq \emptyset$, then $|N_j|\geq \gon(C)>d$. 
By assumption of Case B, we have $N_2\neq \emptyset$, so we conclude that $N_1$ is empty and $|N_2|=\frac{3d}{2}$.
Similarly, if $k=4$ and $N_3\neq \emptyset$, then $|N_3|>d$ and $4d=\deg D=|N_1|+2|N_2|+3|N_3|$, so that $N_1=N_2=\emptyset$ and $|N_3|=\frac{4d}{3}$.
Finally, if $k=4$ and $N_3= \emptyset$,  then $N_2\neq \emptyset$ and $4d=\deg D=|N_1|+2|N_2|$. 
Since $|N_2|>d$ and $2|N_2|$ is a multiple of $d$, we conclude that $N_1=\emptyset$ and $|N_2|=2d$. 

Accordingly, we have to discuss the following three cases:% (i) $k=3$ and $N_2\neq\emptyset$, (ii) $k=4$ and $N_3\neq\emptyset$, (iii) $k=4$ and $N_2\neq\emptyset$.
\begin{itemize}
  \item[(i)] $k=3$ and $N_2\neq\emptyset$;
  \item[(ii)] $k=4$ and $N_3\neq\emptyset$; 
  \item[(iii)] $k=4$ and $N_2\neq\emptyset$.
\end{itemize}
To this aim, we use the following fact.
\begin{claim}\label{claim:d-k}
Consider the points $P_1,\dots,P_d$ in \eqref{eq:P_i}. 
Let $A\subset \{q_1,\ldots,q_s\}$ be a set of points such that for some $j\in \{1,\dots,d\}$, we have $A\cap\Supp(P_j)=\emptyset$ and $A\cap\Supp(P_i)\neq\emptyset$ for any $i\neq j$.
Then $|A|\geq \gon(C)-k$.
\begin{proof}[Proof of Claim \ref{claim:d-k}]
Consider the $(k-1)$-planes $\Lambda_1,\dots,\Lambda_d$ defined in \eqref{eq:Lambda_i} and the linear space $\Span(A)\subset \mathbb{P}^{g-1}$, where $\dim\Span(A)\leq |A|-1$. 
Of course, if $\dim\Span(A)> g-1-k$, then $\Span(A)$ intersects the $(k-1)$-plane $\Lambda_j$.  
If instead $\dim\Span(A)\leq g-1-k$, then $\Span(A)$ meets $\Lambda_j$ because $\Lambda_1,\dots,\Lambda_d$ are $\SP(g-1-k)$ and $A\cap\Lambda_i\neq\emptyset$ for any $i=1,\dots,\widehat{j},\dots,d$.
In any case $\Span(A)\cap \Lambda_j\neq \emptyset$, so that the space $\Span(A,\Lambda_j)$ has dimension at most $|A|+k-2$ and contains (at least) $|A|+k$ distinct points of $C$ (the elements of $A$ and the $k$ points supporting $P_j$).
By the geometric version of the Riemann--Roch theorem, those points define a $\mathfrak{g}^r_{|A|+k}$ on $C$, with $r\geq 1$.
Thus $|A|+k\geq \gon(C)$.
\end{proof}
\end{claim}

\textbf{(i)} If $k=3$ and $N_2\neq\emptyset$, the divisor $D$ has the form $D=2(q_1+\dots+q_s)$, with $s=\frac{3d}{2}$.
We note that $d\neq 2$, otherwise we would have $P_1=P_2=q_1+q_2+q_3$. 
As $s$ is an integer, we may set $d=2c$ and $s=3c$, for some integer $c\geq 2$.

Suppose that $c\neq 2$. 
Given a point $x\in \{q_1,\dots,q_s\}\subset C$, it belongs to the support of two points of $C^{(3)}$, say $P_1$ and $P_2$. 
Then $\Supp(P_1)\cup\Supp(P_2)$ consists of at most 5 distinct points. 
As $s=3c\geq 9$, there exists a point $y\in \{q_1,\dots,q_s\}\smallsetminus\{x\}$ belonging to the support of two other points of $C^{(3)}$, say $P_3$ and $P_4$.
Now, we can choose $n\leq d-5$ distinct points $a_1,\dots,a_n\in \{q_1,\dots,q_s\}\smallsetminus\{x,y\}$, such that the set $A:=\left\{x,y,a_1,\dots,a_n\right\}$ contains a point in the support of $P_i$ for any $i=1,\dots,d-1$ and $A\cap\Supp(P_d)=\emptyset$ (given $x$ and $y$, it suffices to choose---at most---one point in the support of any $P_5,\ldots, P_{d-1}$, avoiding the points of $\Supp(P_d)$).
Thus the set $A$ satisfies the assumption of Claim \ref{claim:d-k}, but $|A|\leq d-3<\gon(C)-3$, a contradiction.
 
On the other hand, suppose that $c=2$. 
Since the curve $E_t$ is irreducible, the points $P_i$ in \eqref{eq:P_i} must be indistinguishable.
Then it is easy to check that, up to reordering indexes, there are only two admissible configurations:
\begin{enumerate}
  \item $P_1=q_1+q_2+q_3$, $P_2=q_1+q_4+q_5$, $P_3=q_2+q_4+q_6$, $P_4=q_3+q_5+q_6$;
  \item $P_1=q_1+q_2+q_3$, $P_2=q_1+q_2+q_4$, $P_3=q_5+q_6+q_3$, $P_4=q_5+q_6+q_4$.
\end{enumerate}
In both cases, we have that the line $\Span(q_4,q_5)$ meets the planes $\Lambda_2$, $\Lambda_3$, and $\Lambda_4$.
Since $\Lambda_1,\dots,\Lambda_4$ are $\SP(g-4)$, we deduce that $\Lambda_1\cap \Span(q_4,q_5)\neq \emptyset$, so that $\Span(\Lambda_1,q_4,q_5)= \Span(q_1,\dots,q_5)\cong \mathbb{P}^3$.
Analogously, by considering the line $\Span(q_4,q_6)$, we deduce that $\Span(q_1,\dots,q_4,q_6)\cong \mathbb{P}^3$.
Therefore, either  $\Span(q_1,\dots,q_4)\cong \mathbb{P}^2$ or $\Span(q_1,\dots,q_6)\cong \mathbb{P}^3$.
In the former case, $|q_1+\dots+q_4|$ is a $\mathfrak{g}^1_4$ on $C$, which is impossible as $\gon(C)>d=2c=4$.

In the latter case, we set $L:=q_1+\dots+q_6$, and $|L|$ turns out to be a $\mathfrak{g}^2_6$ on $C$. 
Since $\gon(C)>4$, $\dim |L-p-q|=0$ for any $p,q\in C$, so that $|L|$ is very ample.
Hence $|L|$ is the unique $\mathfrak{g}^2_6$ on $C$ (see e.g. \cite[Teorema 3.14]{Cil}).
Therefore, by retracing our construction, we have that if $P'\in C^{(3)}$ is a general point, then there exists an effective divisor $L'=q'_1+\dots+q'_s$ such that $|L'|=|L|$ and $P'\leq L'$. 
Thus we obtain a contradiction as $C^{(3)}$ is a threefold, but $\dim |L|=2$.

In conclusion, case (i) does not occur.

\smallskip
\textbf{(ii)} If $k=4$ and $N_3\neq\emptyset$, we have $D=3(q_1+\dots+q_s)$, with $s=\frac{4d}{3}$.
In particular, each $q_j$ belongs to the support of three distinct points among $P_1,\dots,P_d\in C^{(4)}$.
We argue as above, and we set $d=3c$ and $s=4c$, for some integer $c\geq 2$ (if $c$ equaled 1, then we would have $P_1=P_2=P_3=q_1+\dots+q_4$).

Suppose that $c\neq 2$. 
Given a point $x\in \{q_1,\dots,q_s\}$, it belongs to the support of three points of $C^{(4)}$, say $P_1,P_2,P_3$. 
Then $\Supp(P_1)\cup\Supp(P_2)\cup\Supp(P_3)$ consists of at most 10 distinct points. 
As $s=4c\geq 12$, there exists a point $y\in \{q_1,\dots,q_s\}\smallsetminus\{x\}$ belonging to the support of three other points of $C^{(4)}$, say $P_4,P_5,P_6$.
Hence we can choose $n\leq d-7$ distinct points $a_1,\dots,a_n\in \{q_1,\dots,q_s\}\smallsetminus\{x,y\}$, such that the set $A:=\left\{x,y,a_1,\dots,a_n\right\}$ contains a point in the support of $P_i$ for any $i=1,\dots,d-1$  and $A\cap\Supp(P_d)=\emptyset$.
Therefore $A$ satisfies the assumption of Claim \ref{claim:d-k}, with $|A|\leq d-5<\gon(C)-4$, a contradiction.

If instead $c=2$, we have $d=6$ and $s=8$. 
In this setting, we can choose two points $x,y\in \{q_1,\dots,q_8\}$ such that $x+y\leq P_i$ for only one $i\in\{1,\dots,d\}$.
To see this fact, consider the point $q_1$ and, without loss of generality, suppose that $q_1$ belongs to the support of $P_1,P_2,P_3$ and $\Supp(P_1+P_2+P_3)=\left\{q_1,\dots,q_r\right\}$ for some $5\leq r\leq 8$.
If there exists $j\in \{2,\dots,r\}$ such that $q_j$ lies on the support of only one point among $P_1,P_2,P_3$, then $q_1+q_j\leq P_i$ for only one $i\in\{1,2,3\}$, and we may set $(x,y)=(q_1,q_j)$.
If instead any $q_j$ belongs to the support of two points among $P_1,P_2,P_3$, then it is easy to check that $\Supp(P_1+P_2+P_3)=\left\{q_1,\dots,q_5\right\}$, because the divisor $P:=P_1+P_2+P_3$ has degree 12, where $\mult_{p_1}P=3$ and $\mult_{p_j}P\geq 2$ for any $j=2,\dots, r$.
Then we consider the point $q_6$ and, without loss of generality, we suppose that $q_6\in \Supp(P_4)$.
As $s=8$, $\Supp(P_4)$ contains at least an element of $\left\{q_2,\dots,q_5\right\}$, say $q_2$. 
Since $q_2$ belongs also to the support of two points among $P_1,P_2,P_3$, we conclude that $q_2+q_6\leq P_i$ only for $i=4$, and we may set $(x,y)=(q_2,q_6)$.

We point out that $A:=\{x,y\}$ intersects all the sets $\Supp(P_1),\dots,\Supp(P_d)$ but one, because $x$ belongs to the support $P_i$ and two other points of $C^{(4)}$---say $P_2$ and $P_3$---and $y\in \Supp(P_i)$ belongs to the support of two points other than $P_2$ and $P_3$.
Hence $A$ satisfies the assumption of Claim \ref{claim:d-k}, with $|A|=2= d-4<\gon(C)-4$, a contradiction.

It follows that case (ii) does not occur.

\smallskip
\textbf{(iii)} Finally, let $k=4$ and $N_2\neq\emptyset$. 
Then the divisor $D$ has the form $D=2(q_1+\dots+q_s)$, with $s=2d$. 
In particular, each $q_j$ belongs to the support of two distinct points among $P_1,\dots,P_d\in C^{(4)}$.
We note that $d\neq 2$, otherwise we would have $P_1=P_2=q_1+\dots+q_4$.

If $d\geq 8$, we argue as in cases (i) and (ii).
Namely, we consider a point $x\in \{q_1,\dots,q_s\}$ belonging to the support of two points of $C^{(4)}$, say $P_1$ and $P_2$, where $\Supp(P_1)\cup\Supp(P_2)$ consists of at most 7 elements of $\{q_1,\dots,q_s\}$. 
Moreover, as $s=2d\geq 16$, up to reorder indexes, we may consider two distinct points $y,z\in \{q_1,\dots,q_s\}\smallsetminus\{x\}$ such that $y\in \Supp(P_3)\cup\Supp(P_4)$ and $z\in \Supp(P_5)\cup\Supp(P_6)$.
Then we can choose $n\leq d-7$ distinct points $a_1,\dots,a_n\in \{q_1,\dots,q_s\}\smallsetminus\{x,y,z\}$, such that the set $A:=\left\{x,y,z,a_1,\dots,a_n\right\}$ contains a point in the support of $P_i$ for any $i=1,\dots,d-1$ and $A\cap\Supp(P_d)=\emptyset$.
Thus $A$ satisfies the assumption of Claim \ref{claim:d-k}, but $|A|\leq d-4<\gon(C)-4$, a contradiction.

If $3\leq d\leq 7$, we consider the divisor $L:=D_\text{red}=q_1+\dots+q_{2d}\in \Div(C)$.
When $d= 3$, we have that $\Span(q_1,\dots,q_{6})$ has dimension 3. 
Indeed, any point $q_j$ lies on two $3$-planes among $\Lambda_1,\Lambda_2,\Lambda_3$, and by special position property, it must belong to the third one too.
Thus $q_1,\dots,q_{6}$ span a $3$-plane and $L$ is a complete $\mathfrak{g}^2_6$.
Hence $|L-q_6|$ is a complete $\mathfrak{g}^r_5$, where $1\leq r\leq 2$ and $L-q_6=P_1+q_5$. 
We note further that $\dim W^1_5\leq 2$ and $\dim W^2_5\leq 0$ by Martens' theorem (see \cite[Theorem IV.5.1]{ACGH}).
Then we argue similarly to case (i). 
Namely, by retracing our construction, we have that if $P'\in C^{(4)}$ is a general point, then there exists and effective divisor $L'=q'_1+\dots+q'_{5}$ such that $P'\leq L'$ and $|L'|$ is a complete $\mathfrak{g}^{r}_{5}$, with $1\leq r\leq 2$.
Therefore we get a contradiction, since $\dim C^{(4)}=4$, whereas the locus of divisors $L'\in C^{(5)}$ as above has dimension $\dim|L'|\leq 3$.

As far as the remaining values of $d$ are concerned, we need the following.
\begin{claim}\label{claim:456}
If $4\leq d\leq 7$, then $\dim\Span(q_1,\dots,q_s)\leq d$.
\begin{proof}[Proof of Claim \ref{claim:456}]
It follows from \eqref{eq:g} that $g-1\geq 2d-2> 4$. 
Hence the General Position Theorem ensures that any 3-plane $\Lambda_i$ is exactly 4-secant to $C$, and meets $C$ (transversally) only at $\Supp(P_i)$. 
Therefore, without loss of generality, we may assume that $q_1\in\Lambda_1\cap \Lambda_2$ and $q_1\not\in\Lambda_{3},\dots,\Lambda_{d}$.
Let $\pi_1\colon \mathbb{P}^{g-1}\dashrightarrow \bP^{g-2}$ be the projection from $q_1$. 
By Lemma \ref{lem:sottoSP}, for $i=3,\dots,d$, the $3$-planes $\Gamma_i:=\pi_1(\Lambda_i)$ are in special position with respect to the $(g-6)$-planes. 

Suppose that the sequence $\Gamma_3,\dots,\Gamma_d$ is decomposable in the sense of Definition \ref{def:partition}.
This is possible only if $6\leq d\leq 7$, and there exists a part with $2$ elements, say $\{3,4\}$.
Hence $\Gamma_3=\Gamma_4\cong \mathbb{P}^3$, and $\overline{\pi_1^{-1}(\Gamma_3)}=\Span(q_1,\Lambda_3,\Lambda_4)\cong \mathbb{P}^4$ contains at least 6 points of $C$.
Thus $C$ admits a $\mathfrak{g}^1_6$, which is impossible as $\gon(C)>d\geq 6$.

So the sequence $\Gamma_3,\dots,\Gamma_d$ is indecomposable and Theorem \ref{thm:dimSP} ensures that $\dim\Span(\Gamma_3,\dots,\Gamma_d)\leq d-1$.
It follows that $H_1:=\Span(q_1,\Lambda_3,\dots,\Lambda_d)$ satisfies $\dim H_1\leq d$. 
If $q_j\in H_1$ for any $q_j\in \Lambda_1\cap \Lambda_2$, the claim is proved.

By contradiction, suppose that there exists $q_j\in \Lambda_1\cap \Lambda_2$ such that $q_j\not\in H_1$.
Then we can consider the projection $\pi_j\colon \mathbb{P}^{g-1}\dashrightarrow \bP^{g-2}$ and, by arguing as above, we obtain that $H_j:=\Span(q_j,\Lambda_3,\dots,\Lambda_d)$ satisfies $\dim H_j\leq d$.
Since $q_j\not\in H_1$, we deduce that $H:=\Span(\Lambda_3,\dots,\Lambda_d)=H_1\cap H_j$ has dimension $\dim H\leq d-1$. 

If $d=4$, this implies $\Lambda_3=\Lambda_4$, which is impossible because $\Lambda_3$ does not contain all the points of $\Supp(P_4)$.

If $d=5$, Lemma \ref{lem:sottosopraSP} yields that $\Lambda_3,\Lambda_4,\Lambda_5$ are $\SP(g-5)$, because $p_1\not\in H$ and $\Gamma_3,\Gamma_4,\Gamma_5$ are $\SP(g-6)$. 
Since $s=10$, there exists a point $q_j$ in the support of two points among $P_3,P_4,P_5$.
Hence $q_j$ lies on two of the spaces $\Lambda_3,\Lambda_4,\Lambda_5$, and by special position property it has to lie also on the third one, but this is impossible. 

If $d=6$, then $H$ contains at least $2(d-2)=8$ points in $\{q_1,\dots,q_{12}\}$, because for $i=3,\dots,6$, each space $\Lambda_i$ contains exactly 4 points and each point may lie only on two $3$-planes among $\Lambda_1,\dots,\Lambda_6$.
Moreover, if $H$ contained exactly 8 points, then both $\Lambda_1$ and $\Lambda_2$ should contain the remaining 4, a contradiction.
Thus $H$ contains at least 9 points of $C$, and as $\dim H\leq 5$, we have that $C$ admits a $\mathfrak{g}^r_9$ with $r\geq 3$.
Since $\gon(C)>d=6$, the $\mathfrak{g}^r_9$ satisfies the assumption of Lemma \ref{lem:very ample}.
Hence we obtain a contradiction as $r\neq 2$.

If $d=7$, we argue as in the previous case, and we note that $H$ contains at least $2(d-2)=10$ points in $\{q_1,\dots,q_{14}\}$, with $\dim H\leq 6$.
Then $C$ admits a $\mathfrak{g}^r_{10}$ with $r\geq 3$, and since $\gon(C)>d=7$, Lemma \ref{lem:very ample} leads to a contradiction.
\end{proof}
\end{claim}

For any $4\leq d\leq 7$, Claim \ref{claim:456} ensures that the divisor $L=q_1+\dots+q_{2d}$ defines a linear series $|L|$ of dimension $r= \deg L-1-\dim\Span(q_1,\dots,q_{2d})\geq 2d-1-d=d-1\geq 3$.
By combining \eqref{eq:g^r_d} and Claim \ref{claim:456}, we then obtain 
$$
d-1\leq\gon(C)-2\leq \dim\Span(q_1,\dots,q_{2d})\leq d.
$$
If $\dim\Span(q_1,\dots,q_{2d})=d-1$, then $|L|$ turns out to be a $\mathfrak{g}^{d}_{2d}$. 
However, this is impossible by Clifford's theorem and \eqref{eq:g}, as $d\geq 4$.
Hence it follows that $\dim\Span(q_1,\dots,q_{2d})=d$ and $|L|$ is a complete $\mathfrak{g}^{d-1}_{2d}$. 
Moreover, either $L$ is very ample, or there exist $p,q\in C$ such that $\dim|L-p-q|\geq d-2$.

Therefore, in order to conclude case (iii), we analyze separately the following three cases: (iii.a) $L$ is very ample, (iii.b) there exists $p\in C$ such that $\dim|L-p|=d-1$, (iii.c) there exist $p,q\in C$ such that $\dim|L-p|=\dim|L-p-q|=d-2$.

\smallskip
\hspace{1cm}(iii.a) If $L$ is very ample, we claim that $4\leq d\leq 5$.
If indeed $d\geq 6$, then $2d-1=m(d-2)+\varepsilon$ with $m=2$ and $\varepsilon=3$, so that Castelnuovo's bound would imply $g\leq \frac{m(m-1)}{2}(d-2)+m\varepsilon=d+4$.
Hence \eqref{eq:g} implies $2d-1\leq g\leq d+4$, which is impossible for $d\geq 6$.

When $d=5$, $|L|$ is a very ample complete $\mathfrak{g}^{4}_{10}$.
By Castelnuovo's bound we obtain $g\leq 9$. 
Moreover, as $g\geq 2d-1=9$, we deduce that $C$ maps isomorphically to an \emph{extremal} curve of degree $10$ in $\mathbb{P}^4$, i.e. to a curve whose genus attains the maximum in Castelnuovo's bound.
Thus \cite[Theorem III.2.5 and Corollary III.2.6]{ACGH} ensure that $C$ possesses a $\mathfrak{g}^{1}_{3}$, but this contradicts the assumption $\gon(C)>d=5$. 

When $d=4$, $|L|$ is a very ample complete $\mathfrak{g}^{3}_{8}$. 
Let $\phi_{|L|}\colon C\longrightarrow \mathbb{P}^3$ be the corresponding embedding, and consider the \emph{Cayley's number} of $4$-secant lines to a space curve of degree $8$ and genus $g$, defined as
$$
C(8,g,3):=\frac{g^2-21g+100}{2}.
$$
Thanks to Castelnuovo's enumerative formula counting $(2r-2)$-secant $(r-2)$-planes to irreducible curves in $\mathbb{P}^r$ (see e.g. \cite[Theorem 1.2]{ELMS}), we have that if $C(8,g,3)\neq 0$, then $\phi_{|L|}(C)$ admits at least one $4$-secant line.
As $C(8,g,3)\neq0$ for any $g\in \mathbb{N}$, there exists a line $\ell\subset \mathbb{P}^3$ intersecting $\phi_{|L|}(C)$ at $e\geq 4$ points, counted with multiplicities.
Thus the projection from $\ell$ induces a $\mathfrak{g}^{1}_{8-e}$ on $\phi_{|L|}(C)\cong C$, where $8-e\leq 4= d<\gon(C)$, a contradiction. 

\smallskip
\hspace{1cm}(iii.b) Suppose that there exists $p\in C$ such that $\dim|L-p|=d-1$.
Then $|L-p|$ is a $\mathfrak{g}^{d-1}_{2d-1}$, and $\dim|L-p-R|=0$ for any $R\in C^{(d-1)}$, because $\deg(L-p-R)=d<\gon(C)$.
Therefore, Lemma \ref{lem:very ample} ensures that $d-1=2$, which contradicts $d\geq 4$.

\smallskip
\hspace{1cm}(iii.c) Suppose that there exist $p,q\in C$ such that $\dim|L-p|=\dim|L-p-q|=d-2$.
Hence $|L-p-q|$ is a $\mathfrak{g}^{d-2}_{2d-2}$, and $\dim|L-p-q-R|=0$ for any $R\in C^{(d-2)}$, because $\deg(L-p-q-R)=d<\gon(C)$.
Thus Lemma \ref{lem:very ample} implies that $d=4$ and $|L-p-q|$ is a very ample $\mathfrak{g}^{2}_{6}$.
Therefore $g=10$ and $|L|$ is a complete $\mathfrak{g}^{3}_{8}$.

Of course, $C$ is not a plane quintic and, since $\gon(C)>d\geq 4$, $C$ is neither trigonal nor bielliptic.
Hence Mumford's theorem (see \cite[Theorem IV.5.2]{ACGH}) assures that $C$ admits only finitely many complete $\mathfrak{g}^{3}_{8}$.
So we argue similarly to case (i). 
Namely, by retracing our construction, we have that if $P'\in C^{(4)}$ is a general point, then there exists and effective divisor $L'=q'_1+\dots+q'_{8}$ such that $P'\leq L'$ and $|L'|$ is a complete $\mathfrak{g}^{3}_{8}$.
Therefore we get a contradiction, since $\dim C^{(4)}=4$, whereas the locus of divisors $L'\in C^{(8)}$ as above has dimension $\dim|L'|=3$.

\smallskip
By summing up, we proved that Case B.1 does not occur.

\smallskip
\hspace{1cm}\underline{Case B.2}: let $\alpha$ such that $N_{\alpha}=\left\{q_1,\dots,q_s\right\}\neq\emptyset$ and the map $\xi_{\alpha,t}\colon\{t\}\times \mathbb{P}^1\dashrightarrow C^{(s)}$ is constant.
In particular, $Q=q_1+\dots+q_s$ is the only point in the image of $\xi_{\alpha,t}$.
It follows that for general $y\in\mathbb{P}^1$, the points $q_1,\dots,q_s\in C$ are contained in the support of the divisor $D=D_{(t,y)}$ in \eqref{eq:divisorD}.
Since the number of points in $N_{\alpha}\cap \Supp(P_i)$ does not depend on $i=1,\dots,d$, the fiber $\varphi_t^{-1}(t,y)=\{x_1,\dots,x_d\}\subset E_t$ is such that at least one of the points $P_i=f(x_i)$ lies on the $(k-1)$-dimensional subvariety $q_1+C^{(k-1)}\subset C^{(k)}$, at least one lies on $q_2+C^{(k-1)}\subset C^{(k)}$, and so on.
Since $(t,y)$ varies on an open subset of $\{t\}\times \mathbb{P}^1$ and $E_t$ is irreducible, we deduce that $f(E_t)$ is contained in the intersection $\bigcap_{j=1}^s \left(q_j+C^{(k-1)}\right)$, where the points $q_j$ are distinct.

We point out that $s\leq k-2$. 
Indeed, if $s\geq k$, then the intersection above has dimension smaller than 1, so it can not contain the curve $f(E_t)$.
If instead $s= k-1$, then $\bigcap_{j=1}^s \left(q_j+C^{(k-1)}\right)$ is the irreducible curve $q_1+\dots+q_{k-1}+C\cong C$. 
Hence we would have $f(E_t)\cong C$, which is impossible as $\gon\left(f(E_t)\right)=\gon(E_t)=d<\gon(C)$.
Therefore $s\leq k-2$ and $f(E_t)\subset Q+C^{(k-s)}$.

Since $k\in\{3,4\}$, then $1\leq s\leq 2$, and either $Q=q_1$ or $Q=q_1+q_2$. 
In view of the constant map $\xi_{\alpha,t}\colon\{t\}\times \mathbb{P}^1\dashrightarrow C^{(s)}$ such that $(t,y)\longmapsto Q$, we may define the map $\phi\colon T\dashrightarrow C^{(s)}$, sending a general $t\in T$ to the unique point $Q\in C^{(s)}$ in the image of $\xi_{\alpha,t}$.
As $f(E_t)\subset Q+C^{(k-s)}$ and $\cE\stackrel{\pi}{\longrightarrow}T$ is a covering family, the map $\phi$ must be dominant, with general fiber $\phi^{-1}(Q)$ of dimension $\dim T-\dim C^{(s)}=k-s-1$.
Thus the pullback of $\cE\stackrel{\pi}{\longrightarrow} T$ to $\phi^{-1}(Q)\subset T$, i.e.
$$\cE\times_T \phi^{-1}(Q)\longrightarrow \phi^{-1}(Q),$$ 
is a $(k-s-1)$-dimensional family of $d$-gonal curves $E_t$, with $f(E_t)\subset Q+C^{(k-s)}$. 
Furthermore, the latter family does cover $Q+C^{({k-s})}\cong C^{(k-s)}$, otherwise the whole family $\cE\stackrel{\pi}{\longrightarrow} T$ would not cover $C^{(k)}$.
Hence $\covgon\big(C^{(k-s)}\big)\leq \gon\left(f(E_t)\right)=d<\gon(C)$.

If $(k,s)=(3,1)$ or $(k,s)=(4,2)$, then we obtain $\covgon\big(C^{(2)}\big)<\gon(C)$, which contradicts \cite[Theorem 1.6]{B1}.
In particular, this completes the proof of the Theorem \ref{thm:covgon} for $k=3$, showing that if $g\geq 4$ and $\big(g,\gon(C)\big)\neq (4,3)$, then $\covgon(C^{(3)})= \gon(C)$.

Finally, if $(k,s)=(4,1)$, then we have $\covgon(C^{(3)})<\gon(C)$, which contradicts the case $k=3$, and this concludes the proof in the remaining case $k=4$. 
\end{proof}

\begin{remark}[Low genera]\label{rem:low genus}
Let $C$ be a smooth curve of genus $0\leq g\leq k$, i.e. $g$ is outside the range covered by Theorem \ref{thm:covgon}, where $g\geq k+1$.
Under this assumption, the $k$-fold symmetric product $C^{(k)}$ is birational to $J(C)\times \mathbb{P}^{k-g}$.
In particular, 
$$
\covgon \big(C^{(k)}\big)=1\quad\text{for any }0\leq g\leq k-1.
$$

When instead $k=g$, the situation is much more subtle as $C^{(g)}$ is birational to $J(C)$, and its covering gonality is not completely understood.
Since $J(C)$ does not contain rational curves, we deduce that 
$$2\leq \covgon(J(C))=\covgon\big(C^{(g)}\big) \leq \gon(C).$$ 
In particular, we obtain that $\covgon(C^{(2)})=2$ for any curve $C$ of genus $2$.

When $C$ is a very general curve of genus $g$, it follows from \cite[Theorem 2]{Pir} that $J(C)$ does not contain hyperelliptic curves if $g=3$, and \cite[Proposition 4]{B2} extends this fact to any $g\geq 4$.
In particular, if $C$ is a very general curve of genus $g=3,4$, then $\covgon(C^{(g)})=\gon(C)$, as $\gon(C)=\left\lfloor\frac{g+3}{2}\right\rfloor=3$.
In this direction, it would be interesting to understand whether the equality $\covgon(C^{(g)})=\gon(C)$ holds when $C$ is a very general curve of genus $g\geq 5$. 

On the other hand, it is worth noticing that equality $\covgon(C^{(g)})=\gon(C)$ may fail when $C$ is a special curve.
For instance, if the curve $C$ is bielliptic and non-hyperelliptic, then $J(C)$ is covered by elliptic curves and $\covgon(C^{(g)})=2<\gon(C)$.
\end{remark}

\begin{remark}
Concerning other measures of irrationality of an irreducible projective variety $X$, one could investigate the \emph{degree of irrationality} $\irr(X)$, i.e. the least degree of a dominant rational map $X\dashrightarrow \mathbb{P}^{\dim X}$, and the \emph{connecting gonality} $\conngon(X)$, that is the least gonality of a curve $E$ passing through two general points $x,y\in X$, which satisfy $\irr(X)\geq \conngon(X)\geq \covgon(X)$.

In \cite{B1}, various results on $\irr(C^{(k)})$ are proved.
In particular, \cite[Theorem 1.3]{B1} asserts that if $C$ has very general moduli, then $\irr(C^{(2)})\geq g-1$.
Unfortunately, we can not use our techniques in order to obtain a similar result when $k\geq 3$. 
Roughly speaking, the problem is that we can no longer ensure by Corollary \ref{cor:dimSP} that the linear span $\Span(p_1,\dots,p_{dk})$ of the points supporting the divisor $D$ in \eqref{eq:divisorD} is a proper subspace of $\mathbb{P}^{g-1}$.   

On the other hand, if $C$ is a hyperelliptic curve of genus $g\geq 4$, then $\irr(C^{(2)})=\big(\gon(C)\big)^2=4$ by \cite[Theorem 1.2]{B1}.
Furthermore, in the recent paper \cite{CM}, the authors computed the degree of irrationality of the product $C_1\times C_2$ of two curves $C_i$ of genus $g_i\gg \gon(C_i)$, with $i=1,2$.
In particular, they proved that if each $C_i$ is sufficiently general among curves with gonality $\gon(C_i)$, then $\irr(C_1\times C_2)=\gon(C_1)\cdot\gon(C_2)$ (cf. \cite[Theorem B]{CM}), and the same holds when $C_1=C_2$.

In the light of these facts, it is plausible that, by imposing restrictions on the degree of complete linear series on the curve $C$, one can use the same techniques of this paper to compute the degree of irrationality of the $k$-fold symmetric product of $C$, which might agree with bounds analogous to \cite[Proposition 1.1 and Remark 6.7]{B1}.

As far as the connecting gonality of $C^{(k)}$ is concerned, it follows from the results of \cite{BP2} that if $2\leq k\leq 4$ and $C$ is a general curve of genus $g\geq k+4$, then the connecting gonality of $C^{(k)}$ is strictly bigger than the covering gonality, i.e. $\conngon(C^{(k)})> \covgon(C^{(k)})=\gon(C)$ (cf. \cite[Corollary 1.2]{BP2}). 
\end{remark}

%%%%%%%%%%%%%%%%%%%%%%%%%%%%%%%%%%%%%%%
%%%%%%%%%%%%%%%%%%%%%%%%%%%%%%%%%%%%%%%
%%%%%%%%%%%%%%%%%%%%%%%%%%%%%%%%%%%%%%%

\section*{Acknowledgements}

We would like to thank Gian Pietro Pirola for helpful discussions.

%%%%%%%%%%%%%%%%%%%%%%%%%%%%%%%%%%%%%%%
%%%%%%%%%%%%%%%%%%%%%%%%%%%%%%%%%%%%%%%
%%%%%%%%%%%%%%%%%%%%%%%%%%%%%%%%%%%%%%%


\begin{thebibliography}{99}
\bibitem{ACGH} E. Arbarello, M. Cornalba, P. A. Griffiths, J. Harris, \emph{Geometry of Algebraic Curves}, Vol. I, Grundlehren der Mathematischen
    Wissenschaften [Fundamental Principles in Mathematical Sciences] \textbf{267}, Springer-Verlag, New York, 1985.
\bibitem{B2} F. Bastianelli, Remarks on the nef cone on symmetric products of curves, \emph{Manuscripta Math.} \textbf{130} (2009), 113--120.
\bibitem{B1} F. Bastianelli, On symmetric products of curves, \emph{Trans. Amer. Math. Soc.} \textbf{364} (2012), 2493--2519.
\bibitem{BCD} F. Bastianelli, R. Cortini, P. De Poi, The gonality theorem of Noether for hypersurfaces, \emph{J. Algebraic Geom.} \textbf{23} (2014), 313--339.
\bibitem{BDELU} F. Bastianelli, P. De Poi, L. Ein, R. Lazarsfeld, B. Ullery, Measures of irrationality for hypersurfaces of large degree, \emph{Compos. Math.} \textbf{153} (2017), 2368--2393.
\bibitem{BCFS} F. Bastianelli, C. Ciliberto, F. Flamini, P. Supino, Gonality of curves on general hypersurfaces, \emph{J. Math. Pures Appl.} \textbf{125} (2019), 94--118.
\bibitem{BP2} F. Bastianelli, N. Picoco, Moving curves of least gonality on symmetric products of curves, in preparation.
\bibitem{CM} N. Chen, O. Martin, Rational maps from products of curves to surfaces with $p_g=q=0$, \emph{Math. Z.} \textbf{304} (2023), article no. 62, 14 pp.
\bibitem{CS} N. Chen, D. Stapleton, Fano hypersurfaces with arbitrarily large degrees of irrationality, \emph{Forum Math. Sigma} \textbf{8} (2020), article no. e24, 12 pp.
\bibitem{Cil} C. Ciliberto, Alcune applicazioni di un classico procedimento di Castelnuovo, in \emph{Seminari di geometria, 1982--1983 (Bologna, 1982/1983)}, Univ. Stud. Bologna, Bologna, 1984, 17--43.
\bibitem{CMNP} E. Colombo, O. Martin, J. C. Naranjo, G. P. Pirola, Degree of irrationality of a very general abelian variety, \emph{Int. Math. Res. Not. IMRN} \textbf{2022} (2022), 8295--8313.
\bibitem{DM} O. Debarre, L. Manivel, Sur la vari\'et\'e des espaces lin\'eaires contenus dans une intersection compl\`ete, \emph{Math. Ann.} \textbf{312} (1998), 549--574. 
\bibitem{ELMS} D. Eisenbud, H. Lange, G. Martens, F.-O. Schreyer, The Clifford dimension of a projective curve, \emph{Compos. Math.} \textbf{72} (1989), 173--204.
\bibitem{EGH} D. Eisenbud, M. Green, J. Harris, Cayley-Bacharach theorems and conjectures, \emph{Bull. Amer. Math. Soc. (N.S.)} \textbf{33} (1996), 295--324.
\bibitem{GK} F. Gounelas, A. Kouvidakis, Measures of irrationality of the Fano surface of a cubic threefold, \emph{Trans. Amer. Math. Soc.} \textbf{371} (2019), 711--733.
\bibitem {LM} R. Lazarsfeld, O. Martin, Measures of association between algebraic varieties, \emph{Selecta Math. (N.S.)} \textbf{29} (2023), article no. 46, 37 pp.
\bibitem {LU} J. Levinson, B. Ullery, A Cayley-Bacharach theorem and plane configurations, \emph{Proc. Amer. Math. Soc.} \textbf{150} (2022), 4603--4618.
\bibitem {LP} A. F. Lopez, G. P. Pirola, On the curves through a general point of smooth surface in $\mathbb{P}^3$, \emph{Math. Z.} \textbf{219}(1994), 93--106.
\bibitem{Mac} I. G. Macdonald, Symmetric products of an algebraic curve, \emph{Topology} \textbf{1} (1962), 319--343.
\bibitem{Mar} O. Martin, On a conjecture of Voisin on the gonality of very general abelian varieties, \emph{Adv. Math.} \textbf{369} (2020), article no. 107173, 35 pp.
\bibitem{Mir} R. Miranda, \emph{Algebraic Curves and Riemann  Surfaces}, Graduate Studies in Mathematics \textbf{5}, American Mathematical Society, Providence, RI, 1995.
\bibitem{Pic} N. Picoco, Geometry of points satisfying Cayley-Bacharach conditions and applications, \emph{J. Algebra} \textbf{631} (2023), 332--354.
\bibitem{Pir} G. P. Pirola, Curves on generic Kummer varieties, \emph{Duke Math. J.} \textbf{59} (1989), 701--708.
\bibitem{SU} D. Stapleton, B. Ullery, The degree of irrationality of hypersurfaces in various Fano varieties, \emph{Manuscripta Math.} \textbf{161} (2020), 377--408.
\bibitem{V} C. Voisin, On fibrations and measures of irrationality of hyper-K\"ahler manifolds, \emph{Rev. Un. Mat. Argentina} \textbf{64} (2022), 165--197.
\end{thebibliography}
\end{document}